\documentclass[a4paper,11pt,twoside]{amsart}

\usepackage{amsmath}
\usepackage{amsthm}
\usepackage{amsfonts}
\usepackage{amssymb}
\usepackage{mathrsfs}
\usepackage{graphicx}
\usepackage{enumerate}
\usepackage{url}

\usepackage{tikz}

\usetikzlibrary{arrows}
\usetikzlibrary{decorations.markings}

\theoremstyle{plain}
\newtheorem{lemma}{Lemma}[section]
\newtheorem{prop}[lemma]{Proposition}
\newtheorem{theo}[lemma]{Theorem}
\newtheorem{coro}[lemma]{Corollary}
\theoremstyle{remark}
\newtheorem{rem}[lemma]{Remark}

\theoremstyle{definition}
\newtheorem{definition}[lemma]{Definition}
\newtheorem{ex}[lemma]{Example}

\DeclareMathOperator{\Inn}{Inn}
\DeclareMathOperator{\aut}{Aut}
\DeclareMathOperator{\Hom}{Hom}

\newcommand{\id}{\mathrm{Id}}
\newcommand{\C}{\mathscr{C}}
\newcommand{\Ds}{\mathscr{D}}
\newcommand{\N}{\mathbb{N}}
\newcommand{\X}{\mathcal{X}}
\newcommand{\K}{\mathcal{K}}
\newcommand{\loc}{\mathcal{L}}
\newcommand{\Top}{\textnormal{\bf Top}}
\newcommand{\ATop}{\textnormal{\bf ATop}}

\newcommand{\ord}{\textnormal{\bf Ord}}
\newcommand{\cat}{\textnormal{\bf Cat}}
\newcommand{\grp}{\textnormal{\bf Grp}}
\newcommand{\grpd}{\textnormal{\bf Grpd}}
\newcommand{\cst}{c}
\newcommand{\D}{\mathcal{D}}
\newcommand{\E}{\mathcal{E}}

\newcommand{\Su}{\mathbb{S}}
\newcommand{\st}{\mid}
\newcommand{\B}{\mathcal{B}}
\newcommand{\A}{\mathcal{A}}
\newcommand{\Tf}{\mathbb{T}^2}
\newcommand{\Kl}{\mathbb{K}}
\newcommand{\Sp}{\mathcal{S}}
\newcommand{\Fib}{\text{Fib}}
\newcommand{\simp}{\overset{p}{\sim}}
\newcommand{\Ko}{\text{\bf{K}}}
\newcommand{\ev}{\text{ev}}
\newcommand{\T}{\mathcal{T}}

\begin{document}

\title[Fiber bundles over Alexandroff spaces]{Fiber bundles over Alexandroff spaces}

\author{Nicol\'as Cianci}
\address{Facultad de Ciencias Exactas y Naturales \\ Universidad Nacional de Cuyo \\ Mendoza, Argentina.}
\email{nicocian@gmail.com}

\author{Miguel Ottina}
\address{Instituto Interdisciplinario de Ciencias B\'asicas and Facultad de Ciencias Exactas y Naturales \\ Universidad Nacional de Cuyo and CONICET \\ Mendoza, Argentina.}
\email{miguelottina@gmail.com}

\subjclass[2010]{Primary: 55R10, 55R15. Secondary: 54B30.}


\keywords{Fiber bundle, Alexandroff space, Preordered set, Grothendieck construction, Hurewicz fibration.}

\thanks{Research partially supported by grants M044 and 06/M118 of Universidad Nacional de Cuyo. The first author was also partially supported by a CONICET doctoral fellowship.}

\begin{abstract}
We introduce a topological variant of the Grothendieck construction which serves to represent every fiber bundle over an Alexandroff space. Using this result we give a classification theorem for fiber bundles over Alexandroff spaces with T$_0$ fiber and we construct a universal bundle for bundles with T$_0$ fiber over posets which are cofibrant objects of the category of small categories. Moreover, we prove that our construction induces an equivalence of categories between a suitable category of functors and the category of fiber bundles over a fixed Alexandroff space. In addition,  we prove that any fiber bundle over an Alexandroff space is a fibration.
\end{abstract}

\maketitle

\section{Introduction}

Alexandroff spaces are topological spaces which satisfy that any intersection of open subsets is an open subset. Finite topological spaces are perhaps the simplest examples of Alexandroff spaces. It is well known that there exists a functorial correspondence between Alexandroff spaces and preordered sets which preserves the underlying set \cite{alexandroff1937diskrete} . Under this correspondence partially ordered sets correspond to Alexandroff T$_0$--spaces. And since any preordered set can be regarded as a small category, Alexandroff spaces constitute a meeting point of Topology, Combinatorics and Category Theory. 

In addition, McCord proves in \cite{mccord1966singular} that for every simplicial complex $K$ there exists a locally finite T$_0$--space $\X(K)$ together with a weak homotopy equivalence from the geometric realization of $K$ to $\X(K)$. It follows that for each topological space there exists an Alexandroff space which is weak homotopy equivalent to it. Moreover, the category of posets admits a closed model category structure which is Quillen equivalent to the usual model category structure of the category of topological spaces \cite{thomason1980cat,raptis2010homotopy}. 

In this article we introduce a topological variant of the Grothendieck construction for functors from a preordered set to the category of topological spaces, which we call \emph{topological Grothendieck construction} and which coincides with the Grothendieck construction in the case that both of them can be applied. One of the main results of this article states that every fiber bundle over an Alexandroff space $B$ is isomorphic to the topological Grothendieck construction of a morphism-inverting functor whose domain is $B$. Using this result we prove that for any Alexandroff space $B$ and for any T$_0$--space $F$ there exists a canonical bijection between isomorphism classes of functors from $B$ to $\aut(F)$ and fiber bundles over $B$ with fiber $F$, which is induced by the topological Grothendieck construction. As a corollary of this result we obtain that every fiber bundle over a simply-connected Alexandroff space is trivial. In addition, we characterize the functors such that their topological Grothendieck constructions are fiber bundles. Moreover, in section \ref{sect_universal_bundle}, we construct a universal bundle for bundles with T$_0$ fiber over posets which are cofibrant objects of the category of small categories.

Then we apply our results to prove that any fiber bundle over an Alexandroff space is a fibration. Recall that a classical result states that a local fibration\footnote{A continuous function $p\colon E\to B$ is called a \emph{local fibration} if there is an open cover $\{U_\alpha\}_{\alpha\in A}$ of $B$ such that the restriction $p|\colon p^{-1}(U_\alpha)\to U_\alpha$ is a fibration for every $\alpha\in A$. It is immediate that fiber bundles are local fibrations.} over a space $B$ is a fibration, provided that any open cover of $B$ has a numerable refinement \cite{hurewicz1955concept,spanier1981algebraic}. However, since continuous functions from an Alexandroff T$_0$--space to the unit interval are locally constant, non-trivial open covers of connected Alexandroff T$_0$--spaces are not numerable. Hence, this classical result does not apply to fiber bundles over Alexandroff T$_0$--spaces.
 
Finally, we prove that the topological Grothendieck construction yields an equivalence of categories between a suitable category of functors and the category of fiber bundles over a fixed Alexandroff space. The key ingredient for this equivalence of categories is the concept of \emph{weak natural transformation} between functors that we introduce in \ref{def_weak_natural_transformation}, which turns out to be the exact notion of arrows between functors that is needed to obtain the desired equivalence of categories.
 
\section{Preliminaries}

\subsection{Notation} \label{subsec_notation}

We fix the notation that will be used throughout this article.
\begin{itemize}
\item The unit interval $[0,1]$ will be denoted by $I$. 
\item The Sierpinski space will be denoted by $\Sp$. This is the topological space over the set $\{0,1\}$ where the unique non-trivial open subset is the set $\{0\}$.
\item Let $X$ and $Y$ be topological spaces and let $y\in Y$. We define $\cst_y\colon X\to Y$ as the constant map with value $y$.
\item Let $B$ be a topological space. The space of paths in $B$ (with the the compact-open topology) will be denoted by $B^{I}$.
\item Let $X$ and $Y$ be topological spaces and let $f\colon X\to Y$ be a continuous function. Let $\ev_0\colon Y^{I}\to Y$ be the map defined by $\ev_0(\gamma)=\gamma(0)$. We define the space $X\times_f Y^{I}$ as the set $\{(x,\gamma)\in X\times Y^{I}\st\gamma(0)=f(x)\}$ with the subspace topology with respect to $X\times Y^{I}$. Observe that the space $X\times_f Y^{I}$ is the pullback of the diagram $X\xrightarrow{f} Y\xleftarrow{\ev_0} Y^{I}$.

\item Let $B$ be a topological space and let $\alpha,\beta\in B^{I}$. We will write $\alpha\simp\beta$ if $\alpha$ and $\beta$ are path-homotopic. If $\alpha$ and $\beta$ are paths in $B$ such that $\alpha(1)=\beta(0)$, the concatenation of $\alpha$ and $\beta$ will be denoted by $\alpha*\beta$.

\item Let $B$ be a topological space and let $\gamma\in B^{I}$. The inverse path of $\gamma$ will be denoted by $\bar{\gamma}$, and the homotopy class of $\gamma$ will be denoted by $[\gamma]$. In addition, for $0\leq a\leq b\leq 1$, $\gamma_{[a,b]}\colon I\to B$ will denote the (increasing) linear reparametrization of the restriction $\gamma|\colon [a,b]\to B$ of $\gamma$. Observe that if $0\leq a\leq b\leq c\leq 1$ then $\gamma_{[a,b]}*\gamma_{[b,c]}\simp\gamma_{[a,c]}$.

\item Let $X$ be a topological space. The fundamental groupoid of $X$ will be denoted by $\Pi_1(X)$. The composition law in $\Pi_1(X)$ is defined by $[\beta][\alpha]=[\alpha*\beta]$ for paths $\alpha$ and $\beta$ in $X$ such that $\alpha(1)=\beta(0)$.
The fundamental group $\pi_1(X,x_0)$ of $X$ at some $x_0\in X$ is the full subgroupoid of $\Pi_1(X)$ whose unique object is $x_0$.
\item $\cat$ will denote the category of small categories and functors. 
\item $\Top$ will denote the category of topological spaces and continuous functions. In addition, if $B$ is a topological space, the category of objects over $B$ in $\Top$ will be denoted by $\Top/B$.
\item $\Top_0$ will denote the full subcategory of $\Top$ whose objects are the T$_0$ spaces.
\item $\grpd$ will denote the category of groupoids and groupoid homomorphisms and $\grp$ will denote the category of groups and group homomorphisms, considered as a full subcategory of $\grpd$ whose objects are the one-object groupoids.
\item $\ord$ will denote the category of preordered sets and order preserving morphisms. This category will be identified with the category of thin small categories in the usual way, that is, a preordered set $(X,\leq)$ will be considered as a thin category with a unique arrow $x\to y$ (also denoted by $x\leq y$) whenever $x\leq y$, for $x,y\in X$.
\item For a category $\C$ and an object $c_0$ of $\C$, the category $\aut_{\C}(c_0)$ is the subcategory of $\C$ with object $c_0$ and arrows the automorphisms of $c_0$. Note that this category is a one-object groupoid and hence, it will be regarded both as a category and as a group. In particular, for a topological space $X$ and $x_0\in X$, we have that $\pi_1(X,x_0)=\aut_{\Pi_1(X)}(x_0)$.

If $F$ is a topological space, the category $\aut_{\Top}(F)$ will be simply denoted by $\aut(F)$.
\end{itemize}

\subsection{Alexandroff spaces and preordered sets}
\label{subsect_posets}
An \emph{Alexandroff space} is a topological space in which arbitrary intersections of open sets are open. 
The full subcategory of $\Top$ whose objects are the Alexandroff spaces will be denoted by $\ATop$.

Note that for every element $x$ in an Alexandroff space $X$ there exists a \emph{minimal open neighbourhood} of $x$, which is denoted by $U^{X}_x$ (or simply by $U_x$). Namely, $U_x$ is the intersection of every open neighbourhood of $x$. A preorder $\leq$ is defined in an Alexandroff space $X$ by
\begin{displaymath}
x\leq y\text{ if and only if } U_x\subseteq U_y.
\end{displaymath}
The open sets of $X$ are precisely the lower sets of $X$ with respect to $\leq$ and, in particular, $U_x=\{y\in X \st y\leq x\}$ for every $x\in X$.

It is easy to see that this defines a one-to-one correspondence between Alexandroff topologies on a set $X$ and preorder relations in $X$, which induces an isomorphism between $\ATop$ and $\ord$. Hence, every Alexandroff space can be regarded as a preordered set and every continuous function between Alexandroff spaces can be regarded as an order-preserving map. We will make use of this fact throughout this article without further notice. 

Observe also that there are (full) inclusions \[\ATop\hookrightarrow\Top\text{ and }\ord\hookrightarrow\cat.\]

In \cite{mccord1966singular}, M. McCord proved that the minimal open subsets of an Alexandroff space are contractible subspaces. In particular, Alexandroff spaces are locally path-connected, and then the connected components of an Alexandroff space coincide with its path-connected components. In addition, if $X$ is an Alexandroff space then the connected components of $X$ are the connected components of $X$ as a preordered set, that is, if $x,y\in X$ then $x$ and $y$ are on the same connected component of $X$ if and only if there exist $n\in\N$ and $z_0,\ldots,z_n\in X$ such that $z_0=x$, $z_n=y$ and the elements $z_{i-1}$ and $z_i$ are comparable for all $i\in\{1,\ldots,n\}$. Observe that if $X$ is an Alexandroff space and $a,b\in X$ are such that $a\leq b$ then there is a canonical path in $X$ from $a$ to $b$, that will be called $\eta(a\leq b)$, which is defined by
\begin{displaymath}
\eta(a\leq b)(t)=\begin{cases} a&\text{ if $t<1$,}\\ b&\text{ if $t=1$.}\end{cases}
\end{displaymath}

McCord also defines two constructions on topological spaces which are particularly interesting in the context of Alexandroff spaces. These are the \emph{non-Hausdorff cone} and the \emph{non-Hausdorff suspension}. The non-Hausdorff cone of a topological space $X$ is the space $\mathbb{C}X$ over the set $X\cup\{+\}$, with $+$ an element not in $X$, with the topology generated by the open sets of $X$. Namely, the open sets of $\mathbb{C}X$ are those sets that are open in $X$, and the whole set $\mathbb{C}X$.
The non-Hausdorff suspension of a topological space $X$ is the space $\mathbb{S}X$ over the set $X\cup\{+,-\}$, with $+$ and $-$ not in $X$, with the topology generated by the open sets of $X$ and the sets $X\cup \{+\}$ and $X\cup\{-\}$.

These constructions define two functors $\mathbb{C},\mathbb{S}\colon\Top\to\Top$ which restrict to functors $\mathbb{C},\mathbb{S}\colon\ATop\to\ATop$. Note that for an Alexandroff space $X$, the space $\mathbb{C}X$ is the preordered set that is obtained by adding a maximum to $X$, while the space $\mathbb{S}X$ is the preordered set obtained by adding two  incomparable elements that are greater than (and not less than) every element of $X$.

\subsection{Weak homotopy type of Alexandroff spaces}

Recall that the \emph{Kolmogorov quotient} of a space $X$ is the T$_0$--space $\Ko X=X/\sim$ where $\sim$ is the equivalence relation in $X$ that identifies topologically indistinguishable points of $X$, that is, if $a,b\in X$ then $a\sim b$ if and only if for every open subset $U\subseteq X$, $a\in U \Leftrightarrow b\in U$. For every topological space $X$, let $\sigma_X\colon X\to\Ko X$ be the canonical quotient map. Note that for every open subset $U$ of $X$, $\sigma_X^{-1}\sigma_X(U)=U$ and hence $\sigma_X$ is an open map and the initial topology on $X$ with respect to $\sigma_X$ coincides with the topology of $X$. In addition, given a continuous map $f\colon X\to Y$, there exists a unique continuous map $\Ko (f)\colon \Ko X\to \Ko Y$ such that $\Ko(f) \sigma_X=\sigma_Y f$. It is not hard to see that $\Ko$ defines a functor from $\Top$ to $\Top_0$. In addition, if $\iota_0\colon \Top_0\to \Top$ is the inclusion functor then the collection $\{\sigma_X \st X\text{ is a topological space}\}$ defines a natural transformation $\sigma\colon\id_{\Top}\Rightarrow\iota_0\Ko$.

McCord proved in \cite{mccord1966singular} that for every Alexandroff space $X$, any section of the quotient map $\sigma_X$ is a homotopy inverse of $\sigma_X$. Hence, every Alexandroff space is homotopically equivalent to the T$_0$--space $\Ko X$.

Finally, McCord proved that for every Alexandroff T$_0$--space $X$ there exists a natural weak homotopy equivalence $f_X\colon |\K(X)|\to X$, where $\K(X)$ is the simplicial complex with vertices the elements of $X$ and simplices the finite non-empty chains of $X$, and where $|\K(X)|$ denotes the \emph{geometric realization} of the simplicial complex $\K(X)$. Thus, every Alexandroff space $X$ is weakly equivalent to the CW-complex $|\K(\Ko X)|$.

\subsection{Fundamental groupoids of Alexandroff spaces} \label{subsect_fundamental_groupoid}

Recall that there exists a functor $B\colon \cat\to\Top$ that maps every small category to its \emph{classifying space}, that is, to the geometric realization of its simplicial nerve \cite{segal1968classifying}. It is well known that if $X$ is a poset, then $BX$ is naturally homeomorphic to the space $|\K(X)|$. Hence, for every T$_0$ Alexandroff space $X$ there is a natural weak equivalence $\varphi_X\colon BX\to X$. In \cite[Theorem 2.6]{cianci2019coverings} it is shown that this result is in fact true for every Alexandroff space. In particular, $\pi_1(BX,x_0)$ is naturally isomorphic to $\pi_1(X,x_0)$ for every Alexandroff space $X$ and every $x_0\in X$.

Recall also that there exists a functor $\loc\colon\cat\to\grpd$ that maps every small category $\C$ to a groupoid $\loc\C$ which is the localization of $\C$ at its set of morphisms (in the sense of \cite{gabriel1967calculus}). In addition, there exists a functor $\iota_\C\colon \C\to\loc\C$ such that for every morphism-inverting functor $F\colon \C\to\Ds$ there exists a unique functor $\overline{F}\colon \loc \C\to \Ds$ such that $F=\overline{F}\iota_\C$. The groupoid $\loc\C$ and the functor $\iota_\C$ are defined up to a unique canonical isomorphism by this universal property. 

Now, let $X$ be an Alexandroff space and consider the functor $Z_X\colon X\to \Pi_1(X)$ that is the identity on objects and maps every arrow $x\leq x'$ in $X$ to the path-homotopy class $[\eta(x\leq x')]$. Since $\Pi_1(X)$ is a groupoid, this functor is morphism-inverting and thus, there exists a unique functor $\overline{Z_X}\colon \loc X\to \Pi_1(X)$ such that $Z_X=\overline{Z_X}\iota_X$. In \cite[Theorem 3.9]{cianci2019coverings} it is shown that the functor $\overline{Z_X}$ is an isomorphism and that, in fact, the collection $\{Z_B \st B \textnormal{ is an Alexandroff space}\}$ defines a natural isomorphism that can be restricted to a natural isomorphism (of groups) $\aut_{\loc X}(x_0)\cong \pi_1(X,x_0)$ for every $x_0\in X$.

\subsection{The Grothendieck construction}

Recall that the \emph{Grothendieck construction} of a functor $C\colon B\to \cat$ from a small category $B$ to $\cat$, is the category $\int C$ whose objects are pairs $(b,x)$ where $b$ is an object of $B$ and $x$ is an object of $C(b)$. The morphisms in $\int C$ from $(b,x)$ to $(b',x')$ are pairs $(f,g)$ where $f$ is a morphism in $B$ from $b$ to $b'$ and $g$ is a morphism in $C(b')$ from $C(f)(x)$ to $x'$. 
The canonical projection $\pi^C_B\colon \int C\to B$, which maps $(b,x)$ to $b$, is easily seen to be a functor and is usually regarded as an object over $B$.
A natural transformation $\alpha\colon C\Rightarrow D$ between functors $C,D\colon B\to \cat$ induces a functor $\alpha_*\colon \int C\to \int D$ defined by $\alpha_*(b,x)=(b,\alpha_b(x))$ for objects $(b,x)$ of $\int C$ and $\alpha_\ast(f,g)=(f,\alpha_{b'}(g))$ for morphisms $(f,g)$ of $\int C$. 
Furthermore, the Grothendieck construction is actually a functor $\int\colon \cat^{B} \to \cat/B$ from the category $\cat^{B}$ of functors from $B$ to $\cat$ and natural transformations to the category $\cat/B$ of objects over $B$ in $\cat$. 

Now, if $B$ is a preordered set (or equivalently, an Alexandroff space) and $F\colon B\to \cat$ is a functor sending objects of $B$ to preordered sets, then $\int F$ is again a preordered set, where for all $(b,x), (b',x')\in \int F$, we have that $(b,x)\leq (b',x')$ if and only if $b\leq b'$ in $B$ and $F(b\leq b')(x)\leq x'$ in $F(b')$. Thus, $\int F$ is both a category and a topological space.

\section{Topological Grothendieck construction}

In this section we define the \emph{topological Grothendieck construction} for functors $F\colon B\to\Top$, where $B$ is a preordered set, or equivalently, an Alexandroff space. This construction, which extends McCord's non-Hausdorff cone and suspension as well as the non-Hausdorff homotopy colimit of \cite{fernandez2016homotopy}, will play a crucial role in this article.

\begin{definition}
Let $B$ be an Alexandroff space (or equivalently, a preordered set) and let $D\colon B\to\Top$ be a functor. We define 
\begin{displaymath}
\int D = \bigcup\limits_{b\in B}\{b\}\times D(b).
\end{displaymath}
For each $b\in B$ and for each open subset $V$ of $D(b)$ we define 
\begin{displaymath}
J_D(b,V)=\bigcup\limits_{v\in U_b}\{v\}\times D(v\leq b)^{-1}(V).
\end{displaymath}
The set $J_D(b,V)$ will be denoted by $J(b,V)$ when there is no risk of confusion.

Let $\B=\{J(b,V) \st b\in B\textnormal{ and $V$ is an open subset of $D(b)$}\}$. It is not difficult to verify that $\B$ is a basis for a topology on $\int D$ since for all $b,b'\in B$, for all $V$ and $V'$ open subsets of $D(b)$ and $D(b')$ respectively and for all $(\beta,x)\in J(b,V)\cap J(b',V')$ we have that
\begin{displaymath}
(\beta,x)\in J(\beta,D(\beta\leq b)^{-1}(V)\cap D(\beta\leq b')^{-1}(V'))\subseteq J(b,V)\cap J(b',V').
\end{displaymath}
We consider $\int D$ as a topological space with the topology generated by $\B$. The topological space $\int D$ will be called  the \emph{topological Grothendieck construction} of $D$.
\end{definition}

With the notations of above, note that if $B=\varnothing$ then $\int D=\varnothing$. 

\begin{rem} \label{rem_embedding}
Let $B$ be a non-empty Alexandroff space, let $D\colon B\to\Top$ be a functor and let $b\in B$. Let $\iota_b\colon D(b)\to \int D$ be the map defined by $\iota_b(x)=(b,x)$. It is easy to check that $\iota_b$ is a continuous and injective map. Moreover, since for each open subset $V$ of $D(b)$ we have that $V=\iota_b^{-1}(J(b,V))$ it follows that $\iota_b$ is a topological embedding.
\end{rem}

The following proposition states that the definition of the topological Grothendieck construction is compatible with the definition of the Grothendieck construction in the cases that both of them can be applied.

\begin{prop}
Let $B$ be an Alexandroff space and let $D\colon B\to \ord$ be a functor. Thus, $D$ can be regarded both as a functor to $\cat$ and a functor to $\Top$. More precisely, let $\iota_C\colon\ord\to\cat$ be the usual inclusion functor and let $\iota_T\colon\ord\to\Top$ the functor that takes every preordered set to the corresponding Alexandroff space. We consider the compositions $D_C=\iota_C D$ and $D_T=\iota_T D$. Then, the Grothendieck construction $\int D_C$ (which is a preordered set) regarded as an Alexandroff space coincides with the topological Grothendieck construction $\int D_T$. 
\end{prop}

\begin{proof}
Clearly, the underlying sets of $\int D_C$ and $\int D_T$ coincide. Note that for all $b\in B$ and $x\in D(b)$, $J_{D_T}(b,U^{D_T(b)}_x)=U^{\int D_C}_{(b,x)}$. Thus, every open subset of $\int D_C$ is an open subset of $\int D_T$. On the other hand, for every $b\in B$, for every open subset $V$ of $D(b)$ and for every $(\beta,x)\in J_{D_T}(b,V)$ we have that $(\beta,x)\in U^{\int D_C}_{(\beta,x)}=J_{D_T}(\beta,U^{D_T(\beta)}_x)\subseteq J_{D_T}(b,V)$ since $\beta\leq b$ and $U^{D_T(\beta)}_x\subseteq D_T(\beta\leq b)^{-1}(V)$. Thus, for every $b\in B$ and for every open subset $V$ of $D(b)$ the set $J_{D_T}(b,V)$ is an open subset of $\int D_C$. The result follows.
\end{proof}

We give now some simple examples of the topological Grothendieck construction.

\begin{ex}\ 

(1) Let $B$ a non-empty Alexandroff space and let $F$ be a topological space. Let $C_F\colon B\to \Top$ be the constant functor with value $F$. Then $\int C_F=B\times F$ (with the product topology) and $\pi^{C_F}_B\colon B\times F\to B$ is the canonical projection.

(2) Let $X$ be an indiscrete topological space and let $B$ be an Alexandroff space. Let $D\colon B\to \Top$ be a functor such that $D(b)=X$ for all $b\in B$. Then the space $\int D$ is $B\times X$ with the product topology.

(3) Let $X$ be the topological space whose underlying set is $\{a,b,c\}$ and whose topology is $\T_X=\{\varnothing,\{b,c\},X\}$. Let $f_1\colon X\to X$ be the identity map, let $f_2\colon X\to X$ be defined by $f_2(a)=a$, $f_2(b)=b$ and $f_2(c)=b$ and let $f_3\colon X\to X$ be the constant map with value $b$. Let $\Sp$ be the Sierpinski space. For $j\in\{1,2,3\}$ let $F_j\colon\Sp\to\Top$ be the functor defined by $F_j(0)=F_j(1)=X$ and $F_j(0\leq 1)=f_j$.

Then the spaces $\int F_1$ and $\int F_2$ coincide with the space $\Sp\times X$ with the product topology. On the other hand, the space $\int F_3$ is the set $\Sp\times X$ with topology 
\begin{displaymath}
\T=\{\varnothing,\{(0,b),(0,c)\},\{(0,a),(0,b),(0,c)\},\{(0,a),(0,b),(0,c),(1,b),(1,c)\},\Sp\times X\}. 
\end{displaymath}
In particular, $\int F_1$ and $\int F_3$ are not homeomorphic.

(4) Let $X$ be a topological space and let $\ast$ be the singleton. Let $D\colon\Sp\to X$ be the functor defined by $D(0)=X$ and $D(1)=\ast$. Then $\int D$ is the non-Hausdorff cone of $X$.

(5) Let $X$ be a topological space and, as in the previous item, let $\ast$ be the singleton. Let $B$ be the topological space whose underlying set is $\{a,b,c\}$ and whose topology is $\{\varnothing,\{a\},\{a,b\},\{a,c\},\{a,b,c\}\}$. Note that the Hasse diagram of the poset associated to $B$ is
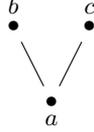
\begin{center}
\begin{tikzpicture}
\tikzstyle{every node}=[font=\footnotesize]
\draw (0.5,0) node(a){$\bullet$} node[below=1]{$a$};
\draw (0,1) node(b){$\bullet$} node[above=1]{$b$};
\draw (1,1) node(c){$\bullet$} node[above=1]{$c$};
\draw (a)--(b);
\draw (a)--(c);
\end{tikzpicture}
\end{center}

Let $D\colon B\to \Top$ be the functor defined by $D(a)=X$ and $D(b)=D(c)=\ast$. Then $\int D$ is the non-Hausdorff suspension of $X$.
\end{ex}

In examples (2) and (3) we can perceive a particular behaviour of the topological Grothendieck construction of a functor $D\colon B\to \Top$ when the spaces $D(b)$ do not satisfy the T$_0$ separation axiom. This is made more explicit in lemma \ref{lemma_Kolmogorov_quotient} and proposition \ref{prop_Kolmogorov_quotient}.

Recall that $\Ko$ denotes the the Kolmogorov quotient functor $\Top\to\Top_0$.

\begin{lemma} \label{lemma_Kolmogorov_quotient}
Let $X$ and $Y$ be topological spaces and let $f,g\colon X\to Y$ be continuous maps. Then $\Ko(f)=\Ko(g)$ if and only if $f^{-1}(V)=g^{-1}(V)$ for every open subset $V$ of $Y$.
\end{lemma}

\begin{proof}
Let $\sigma_X\colon X\to \Ko X$ and $\sigma_Y\colon Y\to \Ko Y$ be the quotient maps. 

Let $V$ be an open subset of $Y$. For each $x\in X$ we have that $\Ko(f)(\sigma_X(x))=\Ko(g)(\sigma_X(x))$, that is, $\sigma_Y(f(x))=\sigma_Y(g(x))$. Therefore,
\begin{displaymath}
x\in f^{-1}(V) \Leftrightarrow f(x)\in V \Leftrightarrow g(x) \in V \Leftrightarrow x\in g^{-1}(V).
\end{displaymath}
Thus, $f^{-1}(V)=g^{-1}(V)$.

To prove the converse, note that for each open subset $W$ of $\Ko Y$ we have that 
\begin{align*}
(\Ko(f)\sigma_X)^{-1}(W) & = (\sigma_Y f)^{-1}(W) = f^{-1}(\sigma_Y ^{-1}(W)) = g^{-1}(\sigma_Y ^{-1}(W)) = (\sigma_Y g)^{-1}(W) = \\ 
&= (\Ko(g)\sigma_X)^{-1}(W)
\end{align*}
Thus, for each $x\in X$ and for each open subset $W$ of $\Ko Y$ we have that $\Ko(f)\sigma_X(x)\in W$ if and only if $\Ko(g)\sigma_X(x)\in W$. Since $\Ko Y$ is a T$_0$--space it follows that, for all $x\in X$, $\Ko(f)\sigma_X(x)=\Ko(g)\sigma_X(x)$. Therefore $\Ko(f)=\Ko(g)$.
\end{proof}

\begin{prop} \label{prop_Kolmogorov_quotient}
Let $B$ be an Alexandroff space and let $F,G\colon B\to \Top$ be functors such that $F(b)=G(b)$ for all $b\in B$. If $\Ko F=\Ko G$ then $\int F=\int G$.
\end{prop}

\begin{proof}
Since $F(b)=G(b)$ for all $b\in B$, it is clear that the underlying sets of $\int F$ and $\int G$ coincide. And from \ref{lemma_Kolmogorov_quotient} it follows that the topologies of $\int F$ and $\int G$ are the same. Thus, $\int F=\int G$.
\end{proof}

\begin{definition}
Let $B$ be a non-empty Alexandroff space (or equivalently, a non-empty preordered set) and let $D\colon B\to\Top$ be a functor. We define $\pi^{D}_B\colon\int D\to B$ as the canonical projection given by $\pi^{D}_B(b,x)=b$. When there is no risk of confusion the projection $\pi^{D}_B$ will be denoted simply by $\pi_B$.

In the case that $B=\varnothing$, we have that $\int D=\varnothing$ and we define $\pi^{D}_B$ as the empty function.
\end{definition}

\begin{rem}\label{rem_int_constant_functor}
Let $B$ a non-empty Alexandroff space and let $F$ be a topological space. Let $C_F\colon B\to \Top$ be the constant functor with value $F$. 
Then $\pi^{C_F}_B$ is the canonical projection of $B\times F$ onto $F$.
\end{rem}

\begin{prop}
Let $B$ be an Alexandroff space and let $D\colon B\to\Top$ be a functor. Then $\pi_B\colon \int D\to B$ is continuous, and hence an object over $B$.
\end{prop}

\begin{proof}
Since $\pi_B^{-1}(U_b)=J(b,D(b))$ for every $b\in B$, the result follows.
\end{proof}

\begin{rem} \label{rem_int_is_functor}
Let $B$ be an Alexandroff space, let $D,E\colon B\to\Top$ be functors and let $\theta\colon D\Rightarrow E$ be a natural transformation. If $b\in B$ and $V$ is an open subset of $E(b)$ then, for all $(\beta,x)\in \int D$,
\begin{align*}
(\beta,\theta_{\beta}(x))\in J_E(b,V) & \Leftrightarrow \beta\leq b \ \land \  E(\beta\leq b)\theta_{\beta}(x) \in V 
\Leftrightarrow \beta\leq b \ \land\  \theta_{b}D(\beta\leq b)(x) \in V \Leftrightarrow \\
& \Leftrightarrow \beta\leq b \ \land\  D(\beta\leq b)(x) \in \theta_{b}^{-1}(V)
\Leftrightarrow (\beta,x)\in J_D(b,\theta_{b}^{-1}(V)).
\end{align*}
It follows that the function $\theta_*\colon\int D\to\int E$ defined by $\theta_*(b,x)=(b,\theta_b(x))$ is continuous and hence a map over $B$ from $\pi^D_B$ to $\pi^E_B$. 

It is easy to see then that the construction $\int$ defines a functor $\int\colon\Top^{B}\to\Top/B$, where $\Top^{B}$ denotes the category of functors from $B$ to $\Top$. In particular, with the previous notations, if $\theta$ is a natural isomorphism then $\theta_*$ is an isomorphism of maps over $B$ from $\pi^D_B$ to $\pi^E_B$.
\end{rem}

The following example shows that the functor $\int\colon\Top^{B}\to\Top/B$ is not essentially surjective.

\begin{ex} \label{ex_topological_Grothendieck_construction_is_not_essentially_surjective}
Let $B$ be the Sierpinski space $\Sp$ (as defined in subsection \ref{subsec_notation}, with unique non-trivial open subset $\{0\}$). Let $E$ be the topological space whose underlying set is $\{a,b,c\}$ and whose topology is $\{\varnothing,\{a\},\{a,b\},\{a,c\},\{a,b,c\}\}$. Let $p\colon E\to B$ be the continuous map given by $p(a)=0$ and $p(b)=p(c)=1$.

Suppose that there exists a functor $D\colon B\to \Top$ such that the projection $\pi_B^D$ is isomorphic to $p$ as maps over $B$. Let $\alpha\colon E\to \int D$ be a homeomorphism such that $\pi_B^D \alpha = p$. Clearly, $D(0)$ is a singleton and $D(1)$ is a topological space with cardinality 2. Moreover, applying \ref{rem_embedding} we obtain that $\alpha$ induces a homeomorphism $\{b,c\}\to D(1)$ and hence $D(1)$ is a discrete space. Let $x$ be the only element of $D(0)$, let $y=D(0\leq 1)(x)$ and let $z$ be the only element of $D(1)-\{y\}$. Note that $J_D(1,\{z\})=\{z\}$ and hence $\{z\}$ is an open subset of $\int D$. But $\alpha^{-1}(z)\in\{b,c\}$ and neither $\{b\}$ nor $\{c\}$ are open subsets of $E$, which entails a contradiction.

Therefore there does not exist any functor $D\colon B\to \Top$ such that the projection $\pi_B^D$ is isomorphic to $p$ as maps over $B$. In particular, the functor $\int\colon\Top^{B}\to\Top/B$ is not essentially surjective.
\end{ex}

The following theorem gives a characterization of the functors which satisfy that the canonical projection maps associated to their topological Grothendieck constructions are fiber bundles.

\begin{theo} \label{theo_characterization_fiber_bundles}
Let $B$ be an Alexandroff space and let $F$ be a topological space. Let $D\colon B\to \Top$ be a functor. For each $b\in B$ let $\sigma_b\colon D(b) \to \Ko D(b)$ be the quotient map.

Then $\pi_B\colon \int D\to B$ is a fiber bundle over $B$ with fiber $F$ if and only if each of the following holds.
\begin{enumerate}[(a)]
\item $\Ko D\colon B\to \Top_0$ is a morphism-inverting functor.
\item For all $b\in B$, $D(b)$ is homeomorphic to $F$.
\item For all $b_1,b_2\in B$ such that $b_1\leq b_2$ and for all $y\in \Ko D(b_2)$, there exists a bijective function $f_{b_1,b_2,y}\colon \sigma_{b_1}^{-1}((\Ko D(b_1\leq b_2))^{-1}(y))\to \sigma_{b_2}^{-1}(y)$.
\end{enumerate}
\end{theo}

\begin{proof}
Suppose first that $\pi_B\colon \int D\to B$ is a fiber bundle over $B$ with fiber $F$. Clearly, item ($b$) holds since, for each $b\in B$, $\pi_B^{-1}(b)=\{b\}\times D(b)$.

Let $b_1,b_2\in B$ such that $b_1\leq b_2$. We will prove that $\Ko D(b_1\leq b_2)$ is a homeomorphism and that for all $y\in \Ko D(b_2)$, there exists a bijective function $f_{b_1,b_2,y}$ as in item ($c$). 

Let $U$ be a trivializing neighbourhood of $b_2$ and let $\varphi_U\colon \pi_B^{-1}(U)\to U\times F$ be a trivialization map. Clearly, the map $\varphi_U$ can be restricted to a homeomorphism $\varphi\colon \pi_B^{-1}(\{b_1,b_2\})\to \{b_1,b_2\}\times F$. In addition, there exist homeomorphisms $\alpha_1\colon D(b_1)\to F$ and $\alpha_2\colon D(b_2)\to F$ such that $\varphi(b_j,x)=(b_j,\alpha_j(x))$ for $j\in\{1,2\}$ and for all $x\in D(b_j)$.

We will prove now that for every open subset $V$ of $D(b_2)$, $D(b_1\leq b_2)^{-1}(V)= \alpha_1^{-1}\alpha_2(V)$. Let $V$ be an open subset of $D(b_2)$. Let $E=\pi_B^{-1}(\{b_1,b_2\})$. Since $J(b_2,V)\cap E$ is an open subset of $E$, we obtain that $\varphi(J(b_2,V)\cap E)$ is an open subset of $\{b_1,b_2\}\times F$. Note that
\begin{align*}
\varphi(J(b_2,V)\cap E) &= \varphi (\{b_1\}\times D(b_1\leq b_2)^{-1}(V) \cup \{b_2\}\times V) = \\
&= \{b_1\}\times \alpha_1(D(b_1\leq b_2)^{-1}(V)) \cup \{b_2\}\times \alpha_2(V).
\end{align*}
Since $\varphi(J(b_2,V)\cap E)$ is an open subset of $\{b_1,b_2\}\times F$ and $b_1\leq b_2$ we obtain that $\alpha_2(V) \subseteq \alpha_1(D(b_1\leq b_2)^{-1}(V))$ and hence $\alpha_1^{-1}\alpha_2(V) \subseteq D(b_1\leq b_2)^{-1}(V)$. On the other hand, let $x\in D(b_1\leq b_2)^{-1}(V)$. Let $x'=D(b_1\leq b_2)(x)$. Note that $\varphi(b_2,x')=(b_2,\alpha_2(x'))\in\{b_1,b_2\}\times \alpha_2(V)$. Since $\alpha_2$ is a homeomorphism we obtain that $\varphi^{-1}(\{b_1,b_2\}\times \alpha_2(V))$ is an open neighbourhood of $(b_2,x')$ in $E$. Thus, there exists an open subset $W$ of $D(b_2)$ such that $(b_2,x')\in J(b_2,W)\cap E\subseteq \varphi^{-1}(\{b_1,b_2\}\times \alpha_2(V))$. Note that $(b_1,x)\in J(b_2,W)$ since $D(b_1\leq b_2)(x)=x'\in W$. Thus, $\varphi(b_1,x)\in \{b_1,b_2\}\times \alpha_2(V)$ and hence $\alpha_1(x)\in \alpha_2(V)$. Then $x\in \alpha_1^{-1}\alpha_2(V)$. Therefore, $D(b_1\leq b_2)^{-1}(V)= \alpha_1^{-1}\alpha_2(V)$.

Thus, applying \ref{lemma_Kolmogorov_quotient} we obtain that $\Ko D(b_1\leq b_2) = \Ko (\alpha_2^{-1}\alpha_1)$ and since $\alpha_1$ and $\alpha_2$ are homeomorphisms it follows that $\Ko D(b_1\leq b_2)$ is a homeomorphism as well. 

Now, let $y\in \Ko D(b_2)$ and let $y'=(\Ko D(b_1\leq b_2))^{-1}(y)$. Hence, $y'=(\Ko (\alpha_2^{-1}\alpha_1))^{-1}(y)$ and since $\sigma_{b_2}\alpha_2^{-1}\alpha_1=\Ko(\alpha_2^{-1}\alpha_1)\sigma_{b_1}$ and $\alpha_2^{-1}\alpha_1$ is a homeomorphism we obtain that $\alpha_2^{-1}\alpha_1(\sigma_{b_1}^{-1}(y'))=\sigma_{b_2}^{-1}(y)$. Thus we may define $f_{b_1,b_2,y}\colon \sigma_{b_1}^{-1}(y')\to \sigma_{b_2}^{-1}(y)$ as a restriction of $\alpha_2^{-1}\alpha_1$.

Conversely, suppose that $(a)$, $(b)$ and $(c)$ hold. Let $b\in B$ and let $h\colon D(b)\to F$ be a homeomorphism. For each $\beta \in B$ such that $\beta\leq b$ and for each $x\in D(\beta)$ let $y_{\beta,x}=\Ko D(\beta\leq b)(\sigma_\beta(x)) \in \Ko D(b)$. Let $\varphi\colon \pi_B^{-1}(U_b)\to U_b\times F$ and $\phi\colon U_b\times F\to \pi_B^{-1}(U_b)$ be defined by
\begin{displaymath}
\varphi(\beta,x)=(\beta,hf_{\beta,b,y_{\beta,x}}(x))
\end{displaymath}
for every $(\beta,x)\in \pi_B^{-1}(U_b)$, and
\begin{displaymath}
\phi(\beta,z)=(\beta,f_{\beta,b,\sigma_b(h^{-1}(z))}^{-1}(h^{-1}(z)))
\end{displaymath}
for every $(\beta,z)\in U_b\times F$, respectively.

Clearly, if $p\colon U_b\times F\to U_b$ is the projection map then the restriction $\pi_B|\colon \pi_B^{-1}(U_b)\to U_b $ of $\pi_B$ is equal to the composition $p_{U_b}\circ \varphi$. In addition, for each $(\beta,x)\in \pi_B^{-1}(U_b)$ we have that
\begin{align*}
\phi (\varphi(\beta,x))&=\phi(\beta,hf_{\beta,b,y_{\beta,x}}(x))=(\beta,f_{\beta,b,\sigma_b(f_{\beta,b,y_{\beta,x}}(x))}^{-1}(f_{\beta,b,y_{\beta,x}}(x)))= \\ &= (\beta,f_{\beta,b,y_{\beta,x}}^{-1}(f_{\beta,b,y_{\beta,x}}(x))) =(\beta,x).
\end{align*}
On the other hand, for each $(\beta,z)\in U_b\times F$, let $x_{\beta,z}=f_{\beta,b,\sigma_b(h^{-1}(z))}^{-1}(h^{-1}(z))\in D(\beta)$. Note that $\sigma_\beta(x_{\beta,z}) = \Ko D (\beta \leq b)^{-1}(\sigma_b(h^{-1}(z)))$ and hence $y_{\beta,x_{\beta,z}}=\sigma_b(h^{-1}(z))$. Thus we have that
\begin{align*}
\varphi(\phi(\beta,z)) &= \varphi(\beta,x_{\beta,z}) =(\beta,hf_{\beta,b,\sigma_b(h^{-1}(z))}(x_{\beta,z}))=(\beta,h(h^{-1}(z)))=(\beta,z).
\end{align*}
Hence, $\varphi$ and $\phi$ are mutually inverse maps.

We will prove now that $\varphi$ is a continuous map. Let $v\in U_b$ and let $W$ be an open subset of $F$. Let $(\beta,x)\in \pi_B^{-1}(U_b)$. Then
\begin{align*}
(\beta,x)\in \varphi^{-1}(U_v\times W) & \Leftrightarrow 
\beta \leq v \,\land\, f_{\beta,b,y_{\beta,x}}(x)\in h^{-1}(W) \Leftrightarrow \\
& \Leftrightarrow \beta \leq v \,\land\, \sigma_b(f_{\beta,b,y_{\beta,x}}(x))\in \sigma_b(h^{-1}(W)) \Leftrightarrow \\
& \Leftrightarrow \beta \leq v \,\land\, y_{\beta,x}\in \sigma_b(h^{-1}(W)) \Leftrightarrow \\
& \Leftrightarrow \beta \leq v \,\land\, \Ko D(\beta\leq b)(\sigma_\beta(x))\in \sigma_b(h^{-1}(W)) \Leftrightarrow \\
& \Leftrightarrow \beta \leq v \,\land\, \Ko D(v\leq b)\Ko D(\beta\leq v)\sigma_\beta(x)\in \sigma_b(h^{-1}(W)) \Leftrightarrow \\
& \Leftrightarrow \beta \leq v \,\land\, \Ko D(v\leq b)\sigma_v D(\beta\leq v)(x)\in \sigma_b(h^{-1}(W)) \Leftrightarrow \\
& \Leftrightarrow \beta \leq v \,\land\, D(\beta\leq v)(x)\in (\Ko D(v\leq b)\sigma_v)^{-1}\sigma_b(h^{-1}(W)) \Leftrightarrow \\
& \Leftrightarrow (\beta,x) \in J(v,(\Ko D(v\leq b)\sigma_v)^{-1}(\sigma_b(h^{-1}(W)))).
\end{align*}
Thus, $\varphi^{-1}(U_v\times W)=J(v,(\Ko D(v\leq b)\sigma_v)^{-1}(\sigma_b(h^{-1}(W))))$. It follows that $\varphi$ is a continuous map.

Now we will prove that $\phi$ is a continuous map. Let $v\leq b$ and let $W$ be an open subset of $D(v)$. Let $(\beta,z)\in U_b\times F$ and let $x_{\beta,z}=f_{\beta,b,\sigma_b(h^{-1}(z))}^{-1}(h^{-1}(z))$. Then
\begin{align*}
(\beta,z)\in\phi^{-1}(J(v,W)) & \Leftrightarrow \beta\leq v \,\land\, D(\beta\leq v)(x_{\beta,z})\in W \Leftrightarrow \\
& \Leftrightarrow \beta\leq v \,\land\, \sigma_v D(\beta\leq v)(x_{\beta,z})\in \sigma_v(W) \Leftrightarrow \\
& \Leftrightarrow \beta\leq v \,\land\,  \Ko D(\beta\leq v)\sigma_\beta(x_{\beta,z})\in \sigma_v(W) \Leftrightarrow \\
& \Leftrightarrow \beta\leq v \,\land\,  \Ko D(v\leq b)^{-1}\Ko D(\beta\leq b)\sigma_\beta(x_{\beta,z})\in \sigma_v(W) \Leftrightarrow \\
& \Leftrightarrow \beta\leq v \,\land\,  \Ko D(v\leq b)^{-1}\sigma_b(h^{-1}(z))\in \sigma_v(W) \Leftrightarrow \\
& \Leftrightarrow \beta\leq v \,\land\,  \sigma_b(h^{-1}(z))\in \Ko D(v\leq b)\sigma_v(W) \Leftrightarrow \\
& \Leftrightarrow (\beta,z)\in U_v\times h(\sigma_b^{-1}(\Ko D(v\leq b)\sigma_v(W))).
\end{align*}
Thus, $\phi^{-1}(J(v,W))=U_v\times h(\sigma_b^{-1}(\Ko D(v\leq b)\sigma_v(W)))$. It follows that $\phi$ is continuous.

Therefore, $\pi_B\colon \int D\to B$ is a fiber bundle over $B$ with fiber $F$.
\end{proof}

\begin{coro} \label{coro_int_D_is_fiber_bundle}
Let $B$ be a connected non-empty Alexandroff space and let $b_0\in B$. Let $D\colon B\to \Top$ be a morphism-inverting functor. Then $\pi_B\colon \int D\to B$ is a fiber bundle over $B$ with fiber $D(b_0)$.
\end{coro}

\begin{proof}
We will prove that items ($a$), ($b$) and ($c$) of \ref{theo_characterization_fiber_bundles} hold. Clearly ($a$) holds since $D$ is a morphism-inverting functor. In addition, since $B$ is connected and $D$ is morphism-inverting we obtain that ($b$) also holds.

Now, let $b_1,b_2\in B$ such that $b_1\leq b_2$ and let $y\in \Ko D(b_2)$. Let $\sigma_{b_1}\colon D(b_1)\to \Ko D(b_1)$ and $\sigma_{b_2}\colon D(b_2)\to \Ko D(b_2)$ be the quotient maps. Since $\sigma_{b_2}D(b_1\leq b_2)=\Ko D(b_1\leq b_2)\sigma_{b_1}$ and $D(b_1\leq b_2)$ is a homeomorphism it follows that $D(b_1\leq b_2)$ induces a bijective function $\sigma_{b_1}^{-1}((\Ko D(b_1\leq b_2))^{-1}(y))\to \sigma_{b_2}^{-1}(y)$. Thus, ($c$) holds.

Therefore, by \ref{theo_characterization_fiber_bundles}, $\pi_B\colon \int D\to B$ is a fiber bundle over $B$ with fiber $D(b_0)$.
\end{proof}

Note that the previous result does not hold if $B$ is not connected since the spaces $D(b)$, $b\in B$, may not be homeomorphic. However, we have the following version of the previous corollary when the Alexandroff space $B$ is not connected. Its proof is similar to that of \ref{coro_int_D_is_fiber_bundle} and will be omitted. 

\begin{coro} \label{coro_functor_to_Aut_F_fiber_bundle}
Let $B$ be an Alexandroff space and let $F$ be a topological space. Let $C\colon B\to \aut(F)$ be a functor and let $\iota\colon\aut(F)\to\Top$ be the inclusion functor. Then the map $\pi^{\iota C}_B\colon \int \iota C\to B$ is a fiber bundle with fiber $F$.
\end{coro}

\section{Classification of fiber bundles over Alexandroff spaces with T$_0$ fiber}

In this section we will prove that any fiber bundle over an Alexandroff space is isomorphic to the projection map associated to the topological Grothendieck construction of a suitable morphism-inverting functor, which can be canonically defined in the case that the fiber is a T$_0$--space. Using this result we will give a classification theorem for fiber bundles over Alexandroff spaces with T$_0$ fiber.

The following proposition will be needed for our purposes.

\begin{prop} \label{lemma_id_times_alpha_sierpinski}
Let $B$ be a topological space with underlying set $\{b_0,b_1\}$ (with $b_0\neq b_1$) such that $b_0 \leq b_1$. Let $F$ be a topological space and let $p_B\colon B\times F\to B$ be the canonical projection. Let $\varphi\colon B\times F\to B\times F$ be an automorphism of the trivial fiber bundle $p_B$.
\begin{enumerate}
\item Let $p_F\colon B\times F\to F$ be the canonical projection. For each $b\in B$ let $j_b\colon F\to B\times F$ be the map defined by $j_b(x)=(b,x)$ and let $\alpha_b=p_F \varphi j_b$. Then $\Ko (\alpha_{b_0})=\Ko (\alpha_{b_1})$.
\item Let $\varphi_0\colon \{b_0\}\times F \to \{b_0\}\times F $ and $\varphi_1\colon \{b_1\}\times F \to \{b_1\}\times F$ be restrictions of $\varphi$ and let $c_{b_1}\colon \{b_0\} \to \{b_1\}$ be the only possible map. Then 
$\Ko((c_{b_1}\times \id_F)\varphi_0)=\Ko(\varphi_1(c_{b_1}\times \id_F))$.
\item If $F$ is a T$_0$--space, then there exists a map $\alpha\colon F\to F$ such that $\varphi=\id_{B}\times \alpha$. Moreover, such a map $\alpha$ is unique and a homeomorphism.
\end{enumerate}
\end{prop}

\begin{proof}\ 

(1) Note that $\varphi(b,x)=(b,\alpha_b(x))$ for all $(b,x)\in B\times F$. In addition, $\alpha_{b_0}$ and $\alpha_{b_1}$ are homeomorphisms.

Let $U$ be an open subset of $F$. Note that $\varphi(B\times U)=(\{b_0\}\times \alpha_{b_0}(U))\cup (\{b_1\}\times\alpha_{b_1}(U))$. Since $\varphi(B\times U)$ is an open subset of $B\times F$ and $b_0\leq b_1$ it follows that $\alpha_{b_1}(U)\subseteq \alpha_{b_0}(U)$. Similarly, $\alpha_{b_1}^{-1}(U)\subseteq \alpha_{b_0}^{-1}(U)$ for each open subset $U$ of $F$. Hence, for every open subset $V$ of $F$,
\begin{displaymath}
V=\alpha_{b_1}(\alpha_{b_1}^{-1}(V))\subseteq \alpha_{b_0}(\alpha_{b_1}^{-1}(V))\subseteq \alpha_{b_0}(\alpha_{b_0}^{-1}(V))=V.
\end{displaymath}
Thus, $\alpha_{b_0}^{-1}(V)=\alpha_{b_1}^{-1}(V)$ for every open subset $V$ of $F$. Then, $\Ko (\alpha_{b_0})=\Ko (\alpha_{b_1})$ by \ref{lemma_Kolmogorov_quotient}.

(2) Let $\alpha_{b_0}$ and $\alpha_{b_1}$ be defined as in the previous item. Let $V$ be an open subset of $\{b_1\}\times F$. Then there exists an open subset $U$ of $F$ such that $V=\{b_1\}\times U$. By the proof of the previous item $\alpha_{b_0}^{-1}(U)=\alpha_{b_1}^{-1}(U)$. Hence,
\begin{align*}
((c_{b_1}\times \id_F)\varphi_0)^{-1}(V)&=\varphi_0^{-1}((c_{b_1}\times \id_F)^{-1}(V))=\varphi_0^{-1}(\{b_0\}\times U)= \{b_0\}\times \alpha_{b_0}^{-1}(U) = \\
&= \{b_0\}\times \alpha_{b_1}^{-1}(U) = (c_{b_1}\times \id_F)^{-1}(\{b_1\}\times \alpha_{b_1}^{-1}(U)) = \\ 
&= (c_{b_1}\times \id_F)^{-1}(\varphi_1^{-1}(V)) = (\varphi_1(c_{b_1}\times \id_F))^{-1}(V).
\end{align*}
Thus $\Ko((c_{b_1}\times \id_F)\varphi_0)=\Ko(\varphi_1(c_{b_1}\times \id_F))$ by \ref{lemma_Kolmogorov_quotient}.

(3) Again, let $\alpha_{b_0}$ and $\alpha_{b_1}$ be defined as in the previous items. By (1), $\Ko (\alpha_{b_0})=\Ko (\alpha_{b_1})$, and since $F$ is a T$_0$--space we obtain that $\alpha_{b_0}=\alpha_{b_1}$. The result follows.
\end{proof}

The following result follows immediately from item (3) of \ref{lemma_id_times_alpha_sierpinski}.

\begin{coro} 
Let $B$ be a connected non-empty Alexandroff space and let $F$ be a T$_0$--space. Let $p_B\colon B\times F\to B$ be the canonical projection and let $\varphi\colon B\times F\to B\times F$ be an automorphism of the trivial fiber bundle $p_B$. Then there exists a map $\alpha\colon F\to F$ such that $\varphi=\id_B\times \alpha$. Moreover, such a map $\alpha$ is unique and a homeomorphism.
\end{coro}

The following theorem is one of the main results of this article.

\begin{theo} \label{theo_fiber_bundles_give_functors}
Let $B$ be an Alexandroff space and let $F$ be a topological space. Let $p\colon E\to B$ be a fiber bundle with fiber $F$. Then there exists a morphism-inverting functor $\D_p\colon B\to \Top$ such that the map $\pi_B\colon \int \D_p\to B$ is a fiber bundle isomorphic to $p$. 
\end{theo}

\begin{proof}
Clearly, we may assume that $B\neq\varnothing$. For each $b\in B$ let $\sigma_b\colon p^{-1}(b)\to \Ko (p^{-1}(b))$ be the quotient map. 

We define a functor $\D\colon B\to \Top$ as follows. For $b\in B$, let $\D(b)=\Ko(p^{-1}(b))$. Now, for $b,b'\in B$ such that $b\leq b'$, choose any trivializing neighbourhood $U$ of $b'$ and let $\varphi_U\colon p^{-1}(U)\to U\times F$ be a trivialization map for $U$. Observe that the map $\varphi_U$ can be restricted to homeomorphisms $\varphi_{U,b}\colon p^{-1}(b) \to \{b\}\times F$ and $\varphi_{U,b'}\colon p^{-1}(b') \to \{b'\}\times F$. Let $\cst_{b'}\colon\{b\}\to \{b'\}$ be the only possible map. Let $\delta_{b,b'}$ be the composition 
\begin{displaymath}
p^{-1}(b)\xrightarrow{\varphi_{U,b}}\{b\}\times F\xrightarrow{\cst_{b'}\times \id_F}\{b'\}\times F\xrightarrow{\varphi_{U,b'}^{-1}}p^{-1}(b').
\end{displaymath}
We define $\D(b\leq b')\colon \Ko(p^{-1}(b))\to \Ko(p^{-1}(b'))$ by $\D(b\leq b')=\Ko(\delta_{b,b'})$. 

We need to show that $\D$ is well-defined, so suppose that $V$ is another trivializing neighbourhood of $b'$ and let $\varphi_V\colon p^{-1}(V)\to V\times F$ be a trivialization map for $V$. Let $\varphi_{V,b}\colon p^{-1}(b) \to \{b\}\times F$, $\varphi_{V,b'}\colon p^{-1}(b') \to \{b'\}\times F$ and $\varphi_{V,\{b,b'\}}\colon p^{-1}(\{b,b'\}) \to \{b,b'\}\times F$ be restrictions of $\varphi_V$ and let $\varphi_{U,\{b,b'\}}\colon p^{-1}(\{b,b'\}) \to \{b,b'\}\times F$ be a restriction of $\varphi_U$. Note that $\varphi_{U,\{b,b'\}}$ and $\varphi_{V,\{b,b'\}}$ are homeomorphisms and that the map $\varphi_{V,\{b,b'\}} (\varphi_{U,\{b,b'\}})^{-1}\colon \{b,b'\}\times F\to \{b,b'\}\times F$ is a bundle automorphism of the projection map $p_{\{b,b'\}}\colon \{b,b'\}\times F\to \{b,b'\}$. By \ref{lemma_id_times_alpha_sierpinski}, the diagram
\begin{displaymath}\label{diagram_definition_Dp}
\begin{tikzpicture}[x=2cm,y=1.5cm]
\tikzstyle{every node}=[font=\footnotesize]
\draw (0,1) node(1){$\Ko (p^{-1}(b))$};
\draw (2,0) node(2){$\Ko (\{b\}\times F)$};
\draw (2,2) node(3){$\Ko (\{b\}\times F)$};
\draw (4,0) node(4){$\Ko (\{b'\}\times F)$};
\draw (4,2) node(5){$\Ko (\{b'\}\times F)$};
\draw (6,1) node(6){$\Ko (p^{-1}(b'))$};
\draw[->] (1) -> (2) node [midway,below,shift={(-0.05,-0.05)}] {$\Ko (\varphi_{V,b})$};
\draw[->] (1) -> (3) node [midway,above,shift={(-0.05,0.05)}] {$\Ko (\varphi_{U,b})$};
\draw[->] (3) -> (2) node [midway,left] {$\Ko (\varphi_{V,b} \varphi_{U,b}^{-1})$};
\draw[->] (2) -> (4) node [midway,below] {$\Ko (\cst_{b'}\times \id_F)$};
\draw[->] (3) -> (5) node [midway,above] {$\Ko (\cst_{b'}\times \id_F)$};
\draw[->] (5) -> (4) node [midway,right] {$\Ko (\varphi_{V,b'} \varphi_{U,b'}^{-1})$};
\draw[->] (5) -> (6) node [midway,above,shift={(0.1,-0.05)}] {$\Ko (\varphi_{U,b'}^{-1})$};
\draw[->] (4) -> (6) node [midway,below,shift={(0.1,-0.05)}] {$\Ko (\varphi_{V,b'}^{-1})$};
\end{tikzpicture}
\end{displaymath}
commutes. Therefore, $\D$ is well-defined. Note that $\D$ is a morphism-inverting functor.

Now we will use the functor $\D$ to define the functor $\D_p$. Let $\sigma_F\colon F\to\Ko F$ be the quotient map. For each $\mu\in\{\#\sigma_F^{-1}(z) \st z\in \Ko F \}$ choose a set $S_\mu$ such that $\# S_\mu = \mu$. For each $b\in B$ and for each $y\in \Ko(p^{-1}(b))$ choose a bijective function $f_{b,y}\colon \sigma_b^{-1}(y)\to S_{\#\sigma_b^{-1}(y)}$. Note that, for all $b,b'\in B$ such that $b\leq b'$ and for all $y\in \Ko(p^{-1}(b))$, the homeomorphism $\delta_{b,b'}$ induces a bijective function $\sigma_b^{-1}(y)\rightarrow \sigma_{b'}^{-1}(\Ko(\delta_{b,b'})(y))=\sigma_{b'}^{-1}(\D(b\leq b')(y))$, and thus $\#\sigma_{b'}^{-1}(\D(b\leq b')(y)) = \#\sigma_b^{-1}(y)$.

We define the functor $\D_p\colon B\to \Top$ as follows. For $b\in B$, let $\D_p(b)=p^{-1}(b)$. Now, for $b,b'\in B$ such that $b\leq b'$, let $\D_p(b\leq b')\colon p^{-1}(b)\to p^{-1}(b')$ be defined by 
\begin{displaymath}
\D_p(b\leq b')(x)=f_{b',\D(b\leq b')(\sigma_b(x))}^{-1}f_{b,\sigma_b(x)}(x).
\end{displaymath}
Note that the map $\D_p(b\leq b')$ is well-defined since $\# \sigma_{b'}^{-1}(\D(b\leq b')(\sigma_b(x))) = \# \sigma_b^{-1}(\sigma_b(x))$ for all $x\in p^{-1}(b)$. Note also that, for each $x\in p^{-1}(b)$, $\sigma_{b'}\D_p(b\leq b')(x)=\D(b\leq b')(\sigma_b(x))$. Thus, $\sigma_{b'}\D_p(b\leq b')=\D(b\leq b')\sigma_b$ and hence $\D_p(b\leq b')$ is a continuous map. 

Clearly, $\D_p(b\leq b)=\id_{p^{-1}(b)}$. Now let $b,b',b''\in B$ such that $b\leq b' \leq b''$ and let $x\in p^{-1}(b)$. Then
\begin{align*}
\D_p(b'\leq b'')\D_p(b\leq b')(x) &= f_{b'',\D(b'\leq b'')(\sigma_{b'}(\D_p(b\leq b')(x)))}^{-1}f_{b',\sigma_{b'}(\D_p(b\leq b')(x))}(\D_p(b\leq b')(x)) = \\
&= f_{b'',\D(b'\leq b'')\D(b\leq b')\sigma_b(x)}^{-1}f_{b',\D(b\leq b')\sigma_b(x)}(\D_p(b\leq b')(x)) = \\
&= f_{b'',\D(b\leq b'')\sigma_b(x)}^{-1} f_{b,\sigma_b(x)}(x) = \D_p(b\leq b'')(x).
\end{align*}
Thus, $\D_p$ is indeed a functor.

We will prove now that $\D_p$ is a morphism-inverting functor. Let $b,b'\in B$ such that $b\leq b'$ and let $g\colon p^{-1}(b')\to p^{-1}(b)$ be defined by $g(x)=f_{b,\D(b\leq b')^{-1}(\sigma_{b'}(x))}^{-1}f_{b',\sigma_{b'}(x)}(x)$. Note that the map $g$ is well-defined since, for all $x\in p^{-1}(b')$,
\begin{align*}
\#\sigma_b^{-1}(\D(b\leq b')^{-1}(\sigma_{b'}(x)))=\#\sigma_{b'}^{-1}(\D(b\leq b')\D(b\leq b')^{-1}(\sigma_{b'}(x)))=\#\sigma_{b'}^{-1}(\sigma_{b'}(x)).
\end{align*}
Note also that $\sigma_b g=\D(b\leq b')^{-1}\sigma_{b'}$. Thus $g$ is a continuous map.
Now, for all $x\in p^{-1}(b)$,
\begin{align*}
g(\D_p(b\leq b')(x)) &= f_{b,\D(b\leq b')^{-1}(\sigma_{b'}(\D_p(b\leq b')(x)))}^{-1}f_{b',\sigma_{b'}(\D_p(b\leq b')(x))}(\D_p(b\leq b')(x)) = \\
&= f_{b,\D(b\leq b')^{-1}\D(b\leq b')\sigma_b(x)}^{-1}f_{b',\D(b\leq b')\sigma_b(x)}(\D_p(b\leq b')(x)) = \\
&= f_{b,\sigma_b(x)}^{-1} f_{b,\sigma_b(x)}(x) = x.
\end{align*}
On the other hand, for all $x\in p^{-1}(b')$,
\begin{align*}
\D_p(b\leq b')(g(x)) &= f_{b',\D(b\leq b')(\sigma_b(g(x)))}^{-1}f_{b,\sigma_b(g(x))}(g(x)) = \\
&= f_{b',\D(b\leq b')(\D(b\leq b')^{-1}\sigma_{b'}(x))}^{-1}f_{b,\D(b\leq b')^{-1}\sigma_{b'}(x)}(g(x)) = \\
&= f_{b',\sigma_{b'}(x)}^{-1}f_{b,\D(b\leq b')^{-1}\sigma_{b'}(x)}f_{b,\D(b\leq b')^{-1}\sigma_{b'}(x)}^{-1}f_{b',\sigma_{b'}(x)}(x)= x.
\end{align*}
Hence, the functions $\D_p(b\leq b')$ and $g$ are mutually inverse. Therefore, $\D_p$ is a morphism-inverting functor.

We will prove now that $\pi_B\colon \int \D_p\to B$ is a fiber bundle isomorphic to $p$. Clearly, it suffices to prove that $\pi_B\colon \int \D_p\to B$ and $p$ are isomorphic as maps over $B$. Let $\phi\colon E\to \int \D_p$ be defined by $\phi(x)=(p(x),x)$. The function $\phi$ is clearly bijective with inverse $\phi^{-1}\colon \int \D_p\to E$ defined by $\phi^{-1}(b,x)=x$ and it is immediate that $\pi_B\phi=p$. Thus, it remains to prove that $\phi$ and its inverse are continuous maps.

We will prove first that $\phi$ is a continuous map. Let $b\in B$ and let $V\subseteq p^{-1}(b)$ be an open subset. We will prove that $\phi^{-1}(J(b,V))$ is an open subset of $E$. Let $U$ be a trivializing neighbourhood of $b$ and let $\varphi\colon p^{-1}(U)\to U\times F$ be a trivialization map. Let $p_F\colon U\times F\to F$ and $p_U\colon U\times F\to U$ be the corresponding projection maps. Since $V\subseteq p^{-1}(b)$ we obtain that $p_U\varphi(V)\subseteq\{b\}$, and hence $\varphi(V)=\{b\}\times p_F\varphi(V)$. Note also that, for all $x\in p^{-1}(U_b)$,
\begin{align*}
\sigma_b\D_p(p(x)\leq b)(x)=\D(p(x)\leq b)\sigma_{p(x)}(x)=\Ko(\delta_{p(x),b}) \sigma_{p(x)}(x)=\sigma_b \delta_{p(x),b} (x).
\end{align*}
Thus, for every $x\in E$,
\begin{align*}
x\in \phi^{-1}(J(b,V)) & \Leftrightarrow (p(x),x)\in J(b,V) \Leftrightarrow p(x)\in U_{b}\,\land\,\D_p(p(x)\leq b)(x)\in V\Leftrightarrow\\
&\Leftrightarrow p(x)\in U_{b} \,\land\, \delta_{p(x),b} (x) \in V \Leftrightarrow\\
&\Leftrightarrow p(x)\in U_{b} \,\land\,(\cst_{b}\times\id_F)\varphi(x)\in\varphi(V)\Leftrightarrow\\
&\Leftrightarrow p_U\varphi(x)\in U_{b}\,\land\,p_F\varphi(x)\in p_F\varphi(V)\Leftrightarrow\\
&\Leftrightarrow \varphi(x)\in U_{b}\times p_F\varphi(V)\Leftrightarrow x\in \varphi^{-1}(U_{b}\times p_F\varphi(V)).
\end{align*}
Hence, $\phi^{-1}(J(b,V))=\varphi^{-1}(U_{b}\times p_F\varphi(V))$. Since the restriction $\varphi|\colon p^{-1}(b)\to\{b\}\times F$ is a homeomorphism it follows that $\varphi(V)$ is an open subset of $\{b\}\times F$ and thus $p_F\varphi(V)$ is an open subset of $F$. Hence $\phi^{-1}(J(b,V))$ is an open subset of $E$ and therefore $\phi$ is a continuous map.

We will prove now that $\phi$ is an open map. Let $V$ be an open subset of $E$. We will prove that, for all $x\in V$,
\[\phi(x)=(p(x),x)\in J(p(x),V\cap p^{-1}(p(x)))\subseteq \phi(V).\]

Let $x\in V$. Clearly $\phi(x)=(p(x),x)\in J(p(x),V\cap p^{-1}(p(x)))$. Let $(b,y)\in J(p(x),V\cap p^{-1}(p(x)))$. Thus $b\leq p(x)$ and $y\in p^{-1}(b)$ since $(b,y)\in\int \D_p$. We will prove that $y\in V$ and thus $(b,y)=\phi(y)\in\phi(V)$ and the result will follow.

Let $U$ be a trivializing neighbourhood of $p(x)$ and let $\varphi\colon p^{-1}(U)\to U\times F$ be a trivialization map for $U$. Let $y'=\delta_{b,p(x)}(y)$. Hence $\varphi(y')=(\cst_{p(x)}\times\id_F)(\varphi(y))$. Thus there exists $x'\in F$ such that $\varphi(y)=(b,x')$ and $\varphi(y')=(p(x),x')$.
Since $(b,y)\in J(p(x),V\cap p^{-1}(p(x)))$, we obtain that $\D_p(b\leq p(x))(y)\in V\cap p^{-1}(p(x))$. Note that 
\begin{align*}
\sigma_{p(x)}\D_p(b\leq p(x))(y)= \D(b\leq p(x))\sigma_{b}(y)=\Ko(\delta_{b,p(x)}) \sigma_{b}(y)=\sigma_{p(x)} \delta_{b,p(x)} (y)=\sigma_{p(x)}(y').
\end{align*}
Since $V\cap p^{-1}(p(x))$ is an open subset of $p^{-1}(p(x))$ we obtain that $y'\in V\cap p^{-1}(p(x))$. Hence \[(p(x),x')=\varphi(y')\in \varphi(V\cap p^{-1}(U)).\]
Since $V\cap p^{-1}(U)$ is an open subset of $p^{-1}(U)$, it follows that $\varphi(V\cap p^{-1}(U))$ is an open subset of $U\times F$. Hence $\varphi(y)=(b,x')\in \varphi(V\cap p^{-1}(U))$ and thus $y\in V$. 
\end{proof}

\begin{rem} \label{rem_Dp_morphism_inverting_T0_fiber}
With the notations of the previous theorem and its proof observe that, if $F$ is a T$_0$--space, then $\#\sigma_F^{-1}(z)=1$ for all $z\in \Ko F$. In addition, for all $b\in B$, the quotient map $\sigma_b$ is a homeomorphism. Thus, for all $b,b'\in B$ such that
$b\leq b'$ we obtain that 
$\sigma_{b'}\D_p(b\leq b') = \D(b\leq b') \sigma_b = \Ko(\delta_{b,b'})\sigma_b = \sigma_{b'}\delta_{b,b'}$
and thus $\D_p(b\leq b') =\delta_{b,b'}$. In addition, since $\Ko X \cong X$ for T$_0$--spaces $X$, by the naturality of the quotient maps $\sigma_X$ it easily follows from the diagram of page \pageref{diagram_definition_Dp} that the definition of the maps $\delta_{b,b'}$ does not depend on the chosen trivialization map. Therefore, if $F$ is a T$_0$--space, the definition of the functor $\D_p$ only depends on the fiber bundle $p$.
\end{rem}

The previous remark motivates the following definition.

\begin{definition}
Let $B$ be an Alexandroff space and let $F$ be any T$_0$--space. Let $p\colon E\to B$ be a fiber bundle with fiber $F$. The morphism-inverting functor $\D_p\colon B\to \Top$ constructed in the previous proof will be called the \emph{canonical representation of the fiber bundle $p$}.
\end{definition}

\begin{rem} \label{rem_independence_on_fiber}
The canonical representation of a fiber bundle with T$_0$--fiber over an Alexandroff space does not depend on the choice of the fiber space. Indeed, let $B$ be an Alexandroff space and let $F$ be any T$_0$--space. Let $p\colon E\to B$ be a fiber bundle with fiber $F$. Suppose that we also regard $p$ as a fiber bundle with fiber $F'$. Clearly, there exists a homeomorphism $\alpha\colon F \to F'$. Let $b,b'\in B$ such that $b\leq b'$. Let $U$ be a trivializing neighbourhood of $b'$ (considering the space $F$ as the fiber) and let $\varphi_U\colon p^{-1}(U) \to U\times F$ be a trivialization map for $U$. Let $\varphi'_U=(\id_U\times \alpha)\varphi_U\colon p^{-1}(U) \to U\times F'$. Note that $\varphi'_U$ is a trivialization map for $U$ regarding now the space $F'$ as the fiber. Following the notations of the proof of theorem \ref{theo_fiber_bundles_give_functors}, we obtain a commutative diagram
\begin{displaymath}
\begin{tikzpicture}[x=1.5cm,y=1.5cm]
\tikzstyle{every node}=[font=\footnotesize]
\draw (0,1) node(1){$p^{-1}(b)$};
\draw (2,0) node(2){$\{b\}\times F'$};
\draw (2,2) node(3){$\{b\}\times F$};
\draw (4,0) node(4){$\{b'\}\times F'$};
\draw (4,2) node(5){$\{b'\}\times F$};
\draw (6,1) node(6){$p^{-1}(b')$};
\draw[->] (1) -> (2) node [midway,below] {$\varphi'_{U,b}$};
\draw[->] (1) -> (3) node [midway,above] {$\varphi_{U,b}$};
\draw[->] (3) -> (2) node [midway,left] {$\id_{\{b\}}\times\alpha$};
\draw[->] (2) -> (4) node [midway,below] {$\cst_{b'}\times \id_F$};
\draw[->] (3) -> (5) node [midway,above] {$\cst_{b'}\times \id_F$};
\draw[->] (5) -> (4) node [midway,right] {$\id_{\{b'\}}\times\alpha$};
\draw[->] (5) -> (6) node [midway,above] {$\varphi_{U,b'}^{-1}$};
\draw[->] (4) -> (6) node [midway,below=4pt] {$(\varphi'_{U,b'})^{-1}$};
\end{tikzpicture}
\end{displaymath}
And since the map $\D_p(b\leq b')$ does not depend on the choice of the trivialization map it follows that it does not depend on the choice of the fiber space either.
\end{rem}

\begin{rem}
Let $B$ be an Alexandroff space, let $p$ and $q$ be fiber bundles over $B$ with T$_0$ fibers and let $f$ be a map over $B$ from $p$ to $q$. For each $b\in B$ let $\overline{f}_b\colon p^{-1}(b)\to q^{-1}(b)$ be defined as a restriction of $f$. Let $\overline{f}=\{\overline{f}_b\colon b\in B\}$. The collection of maps $\overline{f}$ is not, in general, a natural transformation from $\D_p$ to $\D_q$.

For example, let $\Sp$ be the Sierpinski space and let $p\colon \Sp\times\Sp\to\Sp$ be the projection onto the first coordinate. Clearly, the canonical representation of $p$ is the functor $\D_p\colon \Sp\to\Top$ defined by $\D_p(0)=\{0\}\times \Sp$, $\D_p(1)=\{1\}\times \Sp$ and $\D_p(0\leq 1)=\cst_1\times \id_\Sp$, where $\cst_1\colon \{0\}\to\{1\}$ is the only possible map. Let $\alpha\colon \Sp\times\Sp \to \Sp\times\Sp$ be the fiber bundle morphism defined by $\alpha(0,x)=(0,x)$ for all $x\in\Sp$ and $\alpha(1,x)=(1,1)$ for all $x\in\Sp$ (equivalently, $\alpha(b,x)=(b,\max\{b,x\})$ for all $(b,x)\in\Sp\times\Sp$). Let $\overline{\alpha}_0\colon \{0\}\times\Sp\to \{0\}\times\Sp$ and $\overline{\alpha}_1\colon \{1\}\times\Sp\to \{1\}\times\Sp$ be restrictions of $\alpha$. It is easy to check that the collection $\overline{\alpha}=\{\overline{\alpha}_0,\overline{\alpha}_1\}$ is not a natural transformation from $\D_p$ to itself.

Thus, the assignment $p\mapsto \D_p$ that sends each fiber bundle over $B$ with T$_0$ fiber to its canonical representation, can not be extended, in general, to a functor by means of the assignment $f\mapsto \overline{f}$ that sends each morphism $f$ of fiber bundles over $B$ to the collection of maps $\overline{f}$ defined previously.
\end{rem}

As a first application of theorem \ref{theo_fiber_bundles_give_functors} we obtain the following.

\begin{coro} \label{coro_fiber_bundle_with_simply_connected_base}
Let $B$ be a simply connected non-empty Alexandroff space and let $p$ be a fiber bundle over $B$. Then $p$ is a trivial fiber bundle.
\end{coro}

\begin{proof}
By \ref{theo_fiber_bundles_give_functors} there exists a morphism-inverting functor $\D_p\colon B \to \Top$ such that the map $\pi_B\colon \int \D_p\to B$ is a fiber bundle isomorphic to $p$. Let $\iota_B\colon B\to \loc B$ be defined as in subsection \ref{subsect_fundamental_groupoid}. Since $\D_p$ is a morphism-inverting functor, there exists a functor $\overline{\D_p}\colon \loc B\to \Top$ such that $\D_p=\overline{\D_p}\iota_B$. Now, since $B$ is simply connected and the groupoids $\loc B$ and $\Pi_1(B)$ are isomorphic, we obtain that $\loc B$ is an indiscrete category. It follows that the identity functor $\loc B \to \loc B$ is naturally isomorphic to a constant functor. And since $\D_p=\overline{\D_p}\iota_B=\overline{\D_p}\id_{\loc B}\iota_B$ we obtain that the functor $\D_p$ is naturally isomorphic to a constant functor. Hence, from \ref{rem_int_is_functor}, \ref{theo_fiber_bundles_give_functors} and \ref{rem_int_constant_functor} we obtain that $p$ is a trivial fiber bundle.
\end{proof}

We will now give a classification theorem for fiber bundles over Alexandroff spaces with T$_0$ fiber. Propositions \ref{prop_canonical_representations_naturally_isomorphic} and \ref{prop_functor_to_Aut_F_natural_iso} and lemma \ref{lemma_top0_vs_autF} are convenient for this purpose.

\begin{prop} \label{prop_canonical_representations_naturally_isomorphic}
Let $B$ be an Alexandroff space and let $F$ be any T$_0$--space. Let $p$ and $q$ be fiber bundles over $B$ with fiber $F$. Then $p$ and $q$ are isomorphic fiber bundles if and only if their canonical representations $\D_p$ and $\D_q$ are naturally isomorphic.
\end{prop}

\begin{proof}
Suppose that $p$ and $q$ are isomorphic fiber bundles and let $\alpha$ be a homeomorphism such that $q\alpha=p$. For each $b\in B$, let $\theta_b\colon \D_p(b)\to \D_q(b)$ be defined as a restriction of $\alpha$. It is clear that $\theta_b$ is a homeomorphism for all $b\in B$. We will prove that the maps $\theta_b$, $b\in B$, define a natural transformation $\theta\colon \D_p\Rightarrow \D_q$. To this end, let $b_1,b_2\in B$ such that $b_1\leq b_2$. Let $U$ be a trivializing neighbourhood of $b_2$ for the fiber bundle $p$, let $\varphi_U\colon p^{-1}(U)\to U\times F$ be a trivialization map and let $\alpha_U\colon p^{-1}(U)\to q^{-1}(U)$ be defined as a restriction of $\alpha$. It follows that $U$ is also a trivializing neighbourhood for the fiber bundle $q$ and that $\psi_U=\varphi_U\alpha^{-1}_U\colon q^{-1}(U)\to U\times F$ is a trivialization map for $q$. Since the maps $\D_p(b_1\leq b_2)$ and $\D_q(b_1\leq b_2)$ are independent of the choice of the trivialization maps, there is a commutative diagram
\begin{displaymath}
\begin{tikzpicture}[x=1.5cm,y=1.5cm]
\tikzstyle{every node}=[font=\footnotesize]
\draw (0,2) node(1){$p^{-1}(b_1)$};
\draw (6,2) node(4){$p^{-1}(b_2)$};
\draw (2,1) node(2){$\{b_1\}\times F$};
\draw (4,1) node(3){$\{b_2\}\times F$};
\draw (0,0) node(5){$q^{-1}(b_1)$};
\draw (6,0) node(8){$q^{-1}(b_2)$};

\draw[->] (1)--(2) node [midway,above] {$\varphi_{U,b_1}$};
\draw[->] (2)--(3) node [midway,above] {$\cst_{b_2}\times \id_F$};
\draw[->] (3)--(4) node [midway,above] {$\varphi_{U,b_2}^{-1}$};
\draw[->] (1) to node [midway,above] {$\D_p(b_1\leq b_2)$} (4);

\draw[->] (5) -> (2) node [midway,below] {$\psi_{U,b_1}$};
\draw[->] (3) -> (8) node [midway,below] {$\psi_{U,b_2}^{-1}$};
\draw[->] (5) to node [midway,below] {$\D_q(b_1\leq b_2)$} (8);

\draw[->] (1)--(5) node [midway,left]{$\theta_{b_1}$};
\draw[->] (4)--(8) node [midway,right]{$\theta_{b_2}$};
\end{tikzpicture}
\end{displaymath}
where, as in the proof of theorem \ref{theo_fiber_bundles_give_functors}, $\varphi_{U,b_1}$ and $\varphi_{U,b_2}$ denote restrictions of $\varphi_U$ and $\psi_{U,b_1}$ and $\psi_{U,b_2}$ denote restrictions of $\psi_U$. Thus, $\theta\colon \D_p\Rightarrow \D_q$ is a natural transformation.

Conversely, suppose that the canonical representations $\D_p$ and $\D_q$ are naturally isomorphic. By \ref{rem_int_is_functor}, the fiber bundles $\pi_B^{\D_p}\colon \int \D_p \to B$ and $\pi_B^{\D_q}\colon \int \D_q \to B$ are isomorphic. And by \ref{theo_fiber_bundles_give_functors}, the fiber bundle $\pi_B^{\D_p}$ is isomorphic to $p$ and the fiber bundle $\pi_B^{\D_q}$ is isomorphic to $q$. Therefore, $p$ and $q$ are isomorphic fiber bundles.
\end{proof}

\begin{prop} \label{prop_functor_to_Aut_F_natural_iso}
Let $B$ be an Alexandroff space and let $F$ be a T$_0$--space. Let $C\colon B\to \aut(F)$ be a functor and let $\iota\colon\aut(F)\to\Top$ be the inclusion functor. Let $\D_{\pi^{\iota C}_B}$ be the canonical representation of the fiber bundle $\pi^{\iota C}_B\colon \int \iota C\to B$. Then the functors $\iota C$ and $\D_{\pi^{\iota C}_B}$ are naturally isomorphic.
\end{prop}

\begin{proof}
We may assume that $B\neq\varnothing$. Let $b\in B$. By \ref{rem_embedding} the map $i_b\colon F \to \int \iota C$ defined by $i_b(x)=(b,x)$ is a topological embedding. And since $i_b(F)=(\pi^{\iota C}_B)^{-1}(b)=\D_{\pi^{\iota C}_B}(b)$ we obtain that the restriction of $i_b$ to its image defines a homeomorphism $\alpha_b\colon F\to \D_{\pi^{\iota C}_B}(b)$.

We will prove that the collection of arrows $\{\alpha_b \st b\in B\}$ is a natural isomorphism from $\iota C$ to $\D_{\pi^{\iota C}_B}$. Let $b_1,b_2\in B$ such that $b_1\leq b_2$. Let $\varphi\colon (\pi^{\iota C}_B)^{-1}(U_{b_2})\to U_{b_2}\times F$ be the map defined by
$\varphi(\beta,x)=(\beta,C(\beta\leq b_2)(x))$ for every $(\beta,x)\in (\pi^{\iota C}_B)^{-1}(U_{b_2})$. From the proof of \ref{theo_characterization_fiber_bundles} it follows that $\varphi$ is a trivialization map. As in the proof of theorem \ref{theo_fiber_bundles_give_functors}, let $\varphi_{b_1}\colon (\pi^{\iota C}_B)^{-1}(b_1)\to \{b_1\}\times F$ and $\varphi_{b_2}\colon (\pi^{\iota C}_B)^{-1}(b_2)\to \{b_2\}\times F$ be restrictions of $\varphi$. By remark \ref{rem_Dp_morphism_inverting_T0_fiber}, $\D_{\pi^{\iota C}_B}(b_1\leq b_2)=\varphi_{b_2}^{-1}(\cst_{b_2}\times \id_{F})\varphi_{b_1}$ and thus $\D_{\pi^{\iota C}_B}(b_1\leq b_2)\alpha_{b_1}=\alpha_{b_2} C(b_1\leq b_2)$. The result follows.
\end{proof}

The following lemma is an easy consequence of standard arguments in category theory. We include a proof for completeness.

\begin{lemma} \label{lemma_top0_vs_autF}
Let $B$ be an Alexandroff space and let $F$ be a topological space. Let $\D\colon B\to \Top$ be a morphism-inverting functor such that, for each $b\in B$, $\D(b)$ is homeomorphic to $F$. Let $\iota\colon\aut(F)\to\Top$ be the inclusion functor. Then there exists a functor $\E\colon B\to \aut(F)$ such that the functors $\iota \E$ and $\D$ are naturally isomorphic.
\end{lemma}

\begin{proof}
For each $b\in B$ choose a homeomorphism $\gamma_b\colon \D(b)\to F$. Let $\E\colon B\to\aut(F)$ be the functor defined by $\E(b)=F$ for each $b\in B$ and $\E(b_1\leq b_2)=\gamma_{b_2}\D(b_1\leq b_2)\gamma_{b_1}^{-1}$ for $b_1, b_2\in B$ with $b_1\leq b_2$. Clearly, $\E$ is a functor and the homeomorphisms $\gamma_b$, $b\in B$, define a natural isomorphism from $\D$ to $\iota \E$. 
\end{proof}

We give now the aforementioned classification theorem for fiber bundles over Alexandroff spaces with T$_0$ fiber.

\begin{theo} \label{theo_classification_fiber_bundles_Grothendieck_construction}
Let $B$ be an Alexandroff space and let $F$ be any T$_0$--space. Then there exists a canonical bijection between isomorphism classes of fiber bundles over $B$ with fiber $F$ and isomorphism classes of functors from $B$ to $\aut(F)$. This bijection is induced by the canonical representation and its inverse is induced by the topological Grothendieck construction.
\end{theo}

\begin{proof}
Clearly, we may assume that $B\neq\varnothing$. Let $[\Fib_B(F)]$ denote the set of isomorphism classes of fiber bundles over $B$ with fiber $F$ and let $[B,\aut(F)]$ denote the isomorphism classes of functors from $B$ to $\aut(F)$. Let $\iota\colon\aut(F)\to\Top$ be the inclusion functor.

We define a function $\lambda\colon[\Fib_B(F)] \to [B,\aut(F)]$ as follows. Let $p$ be a fiber bundle over $B$ with fiber $F$ and let $[p]$ be its isomorphism class. Let $\D_p \colon B\to\Top$ be the canonical representation of $p$. By lemma \ref{lemma_top0_vs_autF} there exists a functor $\E_p\colon B\to \aut(F)$ such that $\iota \E_p$ and $\D_p$ are naturally isomorphic. We define $\lambda([p])$ as the isomorphism class of the functor $\E_p$. 

We will prove now that $\lambda$ is well-defined. Let $p$ and $q$ be isomorphic fiber bundles over $B$ with fiber $F$, let $\D_p$ and $\D_q$ be the canonical representations of $p$ and $q$ respectively and let $\E_p,\E_q\colon B\to \aut(F)$ be functors such that $\iota \E_p$ is naturally isomorphic to $\D_p$ and $\iota \E_q$ is naturally isomorphic to $\D_q$. By \ref{prop_canonical_representations_naturally_isomorphic}, $\D_p$ and $\D_q$ are naturally isomorphic and hence $\E_p$ and $\E_q$ are naturally isomorphic. Therefore, $\lambda$ is well-defined.

On the other hand, from \ref{coro_functor_to_Aut_F_fiber_bundle} and \ref{rem_int_is_functor} it follows that the topological Grothendieck construction induces a function $\mu\colon[B,\aut(F)]\to[\Fib_B(F)]$. Finally, from \ref{theo_fiber_bundles_give_functors}, \ref{rem_int_is_functor} and \ref{prop_functor_to_Aut_F_natural_iso} it follows that the functions $\lambda$ and $\mu$ are mutually inverse.
\end{proof}

\begin{ex}
Let $S^0$ denote the $0$--sphere. Note that the non-Hausdorff suspension of $S^0$ is the poset $\Su S^0$ defined by the following Hasse diagram, where we have labeled the minimal elements by $a$ and $b$ and the maximal elements by $c$ and $d$.
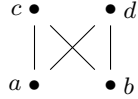
\begin{center}
\begin{tikzpicture}
\tikzstyle{every node}=[font=\footnotesize]
\draw (0,0) node(a){$\bullet$} node[left=1]{$a$};
\draw (1,0) node(b){$\bullet$} node[right=1]{$b$};
\draw (0,1) node(c){$\bullet$} node[left=1]{$c$};
\draw (1,1) node(d){$\bullet$} node[right=1]{$d$};
\draw (a)--(c);
\draw (a)--(d);
\draw (b)--(c);
\draw (b)--(d);
\end{tikzpicture}
\end{center}
By theorem \ref{theo_classification_fiber_bundles_Grothendieck_construction}, the isomorphism classes of fiber bundles over $\Su S^0$ with fiber $\Su S^0$ are in one-to-one correspondence with the isomorphism classes of functors from $\Su S^0$ to $\aut(\Su S^0)$.
Note that $\aut(\Su S^0)=\{\id_{\Su S^0},\tau_{ab},\tau_{cd},\tau_{ab}\tau_{cd}\}$, where $\tau_{xy}$ denotes the transposition that maps $x$ to $y$ (and $y$ to $x$).

For each $\alpha\in \aut(\Su S^0)$ let $\mathcal{G}_\alpha\colon\Su S^0 \to \aut(\Su S^0)$ be the functor defined by $\mathcal{G}_\alpha(a\leq c)=\mathcal{G}_\alpha(a\leq d)=\mathcal{G}_\alpha(b\leq c)=\id_{\Su S^0}$ and $\mathcal{G}_\alpha(b\leq d)=\alpha$. It is not difficult to check that for any functor $G\colon\Su S^0 \to \aut(\Su S^0)$ there exists exactly one element $\alpha\in\aut(\Su S^0)$ such that $G$ is naturally isomorphic to $\mathcal{G}_\alpha$.

Therefore, there exist exactly four isomorphism classes of fiber bundles over $\Su S^0$ with fiber $\Su S^0$ which correspond to the functors $\mathcal{G}_\alpha$ for $\alpha\in \aut(\Su S^0)$.

It is interesting to observe that the total spaces of those fiber bundles are homeomorphic to the spaces $\Tf_{0,0}$, $\Tf_{1,1}$, $\Kl_{1,0}$ and $\Kl_{0,1}$ given in \cite{cianci2018splitting}, which are the \emph{minimal finite models} of the torus and the Klein bottle.
\end{ex}

\begin{rem}
Theorem \ref{theo_classification_fiber_bundles_Grothendieck_construction} may not hold if the fiber $F$ is not a T$_0$--space. Indeed, let $F$ be the indiscrete space with underlying set $\{1,2\}$. As in the previous example, let $\Su S^0$ be the non-Hausdorff suspension of the $0$--sphere $S^0$, let $a$ and $b$ be the minimal elements of $\Su S^0$ and let $c$ and $d$ be its maximal elements.

Let $f_1\colon F\to F$ be the identity map and let $f_2\colon F\to F$ be defined by $f_2(1)=2$ and $f_2(2)=1$. For $j\in\{1,2\}$ let $D_j\colon\Su S^0\to\aut(F)$ be the functor defined by $D_j(a\leq c)=D_j(a\leq d)=D_j(b\leq c)=\id_F$ and $D_j(b\leq d)=f_j$. A simple argument shows that the functors $D_1$ and $D_2$ are not naturally isomorphic. However, it is not difficult to verify that the only fiber bundle over $\Su S^0$ with fiber $F$ is the trivial bundle. This can be proved using the definition of fiber bundle and standard general topology arguments or applying \ref{theo_fiber_bundles_give_functors}, \ref{lemma_top0_vs_autF}, \ref{rem_int_is_functor}, \ref{rem_int_constant_functor} and \ref{prop_Kolmogorov_quotient}.
\end{rem}

\section{Construction of a universal bundle} \label{sect_universal_bundle}

In this section we will construct a universal bundle with fiber a given T$_0$--space.

\begin{prop} \label{prop_canonical_representation_of_pullback}
Let $X$ and $B$ be Alexandroff spaces and let $f\colon X\to B$ be a continuous map. Let $D\colon B\to \Top$ be a functor. Then there is a pullback diagram
\begin{center}
\begin{tikzpicture}[x=2cm,y=2cm]
\draw (0,0) node (X) {$X$};
\draw (0,1) node (Df) {$\int Df$};
\draw (1,0) node (B) {$B$};
\draw (1,1) node (D) {$\int D$};
\draw (0.5,0.5) node {\textnormal{pull}};

\draw[->] (X)--(B) node [midway,below] {$f$};
\draw[->] (Df)--(D) node [midway,above] {$g$};
\draw[->] (Df)--(X) node [midway,left] {$\pi^{Df}_X$};
\draw[->] (D)--(B) node [midway,right] {$\pi^D_B$};
\end{tikzpicture}
\end{center}
where the map $g\colon \int Df \to \int D$ is defined by $g(x,y)=(f(x),y)$.
\end{prop}

\begin{proof}
First, we will prove that the map $g$ is continuous. Let $b\in B$ and let $V$ be an open subset of $D(b)$. It is not difficult to verify that
\begin{displaymath}
g^{-1}(J_D(b,V))=\bigcup_{x\in f^{-1}(U_b)}J_{Df}(x,D(f(x)\leq b)^{-1}(V))
\end{displaymath}
Thus, $g$ is a continuous map.

Observe that $\pi^D_B g = f \pi^{Df}_X$.

Now, let $Z$ be a topological space and let $\alpha\colon Z\to X$ and $\beta\colon Z\to \int D$ be continuous maps such that $f\alpha=\pi^{D}_B\beta$. Note that for each $z\in Z$, $\beta(z)=(\beta_1(z),\beta_2(z))$ with $\beta_1(z)\in B$ and $\beta_2(z)\in D(\beta_1(z))$. And since $f\alpha=\pi^{D}_B\beta$ we obtain that $\beta_1(z)=f(\alpha(z))$ for all $z\in Z$. Let $\gamma\colon Z\to \int Df$ be defined by $\gamma(z)=(\alpha(z),\beta_2(z))$. It is clear that the function $\gamma$ is well-defined and satisfies $\pi^{Df}_X\gamma=\alpha$ and $g\gamma=\beta$. Moreover, $\gamma$ is the only function from $Z$ to $\int Df$ which satisfies this property. It remains to prove that $\gamma$ is continuous.

Let $x\in X$ and let $V$ be an open subset of $Df(x)$. It is not difficult to verify that
\begin{displaymath}
\gamma^{-1}(J_{Df}(x,V))=\alpha^{-1}(U_x)\cap \beta^{-1}(J_D(f(x),V))
\end{displaymath}
Therefore $\gamma$ is a continuous map.
\end{proof}

The following result is an easy consequence of the previous proposition.

\begin{prop}
Let $X$ and $B$ be connected Alexandroff spaces and let $f,g\colon X\to B$ be continuous maps such that $f\leq g$. Let $D\colon B\to \Top$ be a morphism-inverting functor. Then $\pi_X^{Df}\colon \int Df \to X $ and $\pi_X^{Dg}\colon \int Dg \to X$ are isomorphic fiber bundles.
\end{prop}

\begin{proof}
Note that the continuous maps $f$ and $g$ can be regarded as functors, and since $f \leq g$, there exists a natural transformation $\phi\colon f\Rightarrow g$ (where, for each $x$ in $X$, $\phi_x$ is the only morphism from $f(x)$ to $g(x)$). Thus, we have a natural transformation $D\phi\colon Df\Rightarrow Dg$, which is a natural isomorphism since $D$ is a morphism-inverting functor. The result then follows from \ref{rem_int_is_functor}.
\end{proof}

\begin{prop} \label{prop_coro_canonical_representation_of_pullback}
Let $B$ and $X$ be Alexandroff spaces, let $p\colon E\to B$ be a fiber bundle with T$_0$ fiber and let $f\colon X\to B$ be a continuous map. Let $\D_p\colon B\to \Top$ be the canonical representation of the fiber bundle $p$. Consider the pullback diagram
\begin{center}
\begin{tikzpicture}[x=2cm,y=2cm]
\draw (0,0) node (X) {$X$};
\draw (0,1) node (P) {$P$};
\draw (1,0) node (B) {$B$};
\draw (1,1) node (E) {$E$};
\draw (0.5,0.5) node {\textnormal{pull}};

\draw[->] (X)--(B) node [midway,below] {$f$};
\draw[->] (P)--(E) ;
\draw[->] (P)--(X) node [midway,left] {$q$};
\draw[->] (E)--(B) node [midway,right] {$p$};
\end{tikzpicture}
\end{center}
Then, the canonical representation $\D_q$ is naturally isomorphic to $\D_p f$.
\end{prop}

\begin{proof}
By \ref{theo_fiber_bundles_give_functors}, the fiber bundle $\pi^{\D_p}_B\colon \int \D_p \to B$ is isomorphic to $p$ and the fiber bundle $\pi^{\D_q}_X\colon \int \D_q \to X$ is isomorphic to $q$. Thus, we have a pullback diagram
\begin{center}
\begin{tikzpicture}[x=2cm,y=2cm]
\draw (0,0) node (X) {$X$};
\draw (0,1) node (P) {$\int \D_q$};
\draw (1,0) node (B) {$B$};
\draw (1,1) node (E) {$\int \D_p$};
\draw (0.5,0.5) node {\textnormal{pull}};

\draw[->] (X)--(B) node [midway,below] {$f$};
\draw[->] (P)--(E) ;
\draw[->] (P)--(X) node [midway,left] {$\pi^{\D_q}_X$};
\draw[->] (E)--(B) node [midway,right] {$\pi^{\D_p}_B$};
\end{tikzpicture}
\end{center}
Hence, by \ref{prop_canonical_representation_of_pullback}, the fiber bundles $\pi^{\D_q}_X$ and $\pi^{\D_p f}_X$ are isomorphic. 

Let $\iota\colon \aut(F)\to \Top$ be the inclusion functor. By \ref{lemma_top0_vs_autF} there exist functors $\E_p,\E_q\colon X\to \aut(F)$ such that $\iota \E_p$ is naturally isomorphic to $\D_p f$ and $\iota \E_q$ is naturally isomorphic to $\D_q$. By \ref{rem_int_is_functor}, the fiber bundle $\pi^{\iota \E_p}_X$ is isomorphic to $\pi^{\D_p f}_X$ and the fiber bundle $\pi^{\iota \E_q}_X$ is isomorphic to $\pi^{\D_q}_X$. Thus, $\pi^{\iota \E_p}_X$ and $\pi^{\iota \E_q}_X$ are isomorphic fiber bundles and hence, from \ref{theo_classification_fiber_bundles_Grothendieck_construction}, we obtain that $\E_p$ and $\E_q$ are naturally isomorphic. Thus, the functors $\D_q$ and $\D_p f$ are naturally isomorphic.
\end{proof}

\begin{prop} \label{prop_equivalences_isomorphic_pullback_bundles}
Let $X$ and $Y$ be non-empty Alexandroff spaces and let $F$ be a T$_0$--space. Let $p\colon E\to Y$ be a fiber bundle with fiber $F$ and let $f,g\colon X\to Y$ be continuous maps. Let $p_f\colon E_f\to X$ be the pullback of $p$ along $f$ and let $p_g\colon E_g\to X$ be the pullback of $p$ along $g$.

Let $\D_p\colon Y \to \Top$ be the canonical representation of the fiber bundle $p$, let $\iota\colon \aut(F)\to \Top$ be the inclusion functor and let $\E_p\colon Y \to \aut(F)$ be a functor which satisfies that $\iota \E_p$ and $\D_p$ are naturally isomorphic.

Let $A\subseteq X$ be such that for each connected component $C$ of $X$, $\#(A\cap C)=1$. For each $x_0\in A$ let $(\E_p f)_{\ast,x_0},(\E_p g)_{\ast,x_0}\colon \pi_1(X,x_0) \to \pi_1(\aut(F),F)\cong\aut(F)$ be the group homomorphisms induced by the functors $\E_p f$ and $\E_p g$ respectively.

The following are equivalent:
\begin{enumerate}
\item The fiber bundles $p_f$ and $p_g$ are isomorphic.
\item The functors $\D_p f$ and $\D_p g$ are naturally isomorphic.
\item For each $x_0\in A$ there exists $\nu_{x_0}\in\Inn(\aut(F))$ such that $(\E_p g)_{\ast,x_0}=\nu_{x_0} (\E_p f)_{\ast,x_0}$.
\end{enumerate}
\end{prop}

\begin{proof}
By \ref{prop_canonical_representations_naturally_isomorphic}, the fiber bundles $p_f$ and $p_g$ are isomorphic if and only if their canonical representations $\D_{p_f}$ and $\D_{p_g}$ are naturally isomorphic functors. By \ref{prop_coro_canonical_representation_of_pullback} this holds if and only if the functors $\D_p f$ and $\D_p g$ are naturally isomorphic.

Now, observe that the functors $\D_p f$ and $\D_p g$ are naturally isomorphic if and only if the functors $\E_p f$ and $\E_p g$ are naturally isomorphic. Let $\loc X$ be the localization of $X$ and let $\iota_X\colon X\to \loc X$ be defined as in subsection \ref{subsect_fundamental_groupoid}. Let $\overline{\E}_f,\overline{\E}_g\colon \loc X\to\aut(F)$ be functors such that $\overline{\E}_f \iota_X=\E_p f$ and $\overline{\E}_g \iota_X=\E_p g$. Let $\A$ be the full subcategory of $\loc X$ whose set of objects is $A$ and let $i_{\A}\colon \A \to \loc X$ be the inclusion functor. Observe that the functors $\E_p f$ and $\E_p g$ are naturally isomorphic if and only if the functors $\overline{\E}_f$ and $\overline{\E}_g$ are naturally isomorphic. And this holds if and only if the functors $\overline{\E}_f i_{\A}$ and $\overline{\E}_g  i_{\A}$ are naturally isomorphic since $i_{\A}$ is an equivalence of categories. 

For each $x_0\in A$ let $\aut_{\loc X}(x_0)$ be the full subcategory of $\loc X$ whose only object is $x_0$ and let $i_{x_0}\colon \aut_{\loc X}(x_0) \to \A$ be the inclusion functor. Observe that the functors $\overline{\E}_f i_{\A}$ and $\overline{\E}_g  i_{\A}$ are naturally isomorphic if and only if for each $x_0\in A$ there exists $\nu_{x_0}\in\Inn(\aut(F))$ such that $\overline{\E}_g i_{\A} i_{x_0} = \nu_{x_0}\overline{\E}_f  i_{\A} i_{x_0}$.

Let $\cat_*$ be the category of pointed small categories and basepoint-preserving functors. Recall that there exists a natural isomorphism $\zeta\colon \aut_{\loc(\cdot)}(\cdot)\cong \pi_1(B(\cdot),\cdot)\colon \cat_*\to\grp$ \cite{quillen1973higher} (see also \cite[Corollary 3.5]{cianci2019coverings}). For each $x_0\in A$ let $\alpha=\zeta_{(X,x_0)}^{-1}\colon \pi_1(X,x_0)\to \aut_{\loc X}(x_0)$. From the naturality of $\zeta$ it follows that $(\E_p f)_{\ast,x_0}=\overline{\E}_f  i_{\A} i_{x_0}  \alpha$ and $(\E_p g)_{\ast,x_0}=\overline{\E}_g  i_{\A} i_{x_0}  \alpha$. Thus, there exists $\nu_{x_0}\in\Inn(\aut(F))$ such that $\overline{\E}_g i_{\A} i_{x_0} = \nu_{x_0}\overline{\E}_f  i_{\A} i_{x_0}$ if and only if there exists $\nu_{x_0}\in\Inn(\aut(F))$ such that $(\E_p g)_\ast=\nu_{x_0} (\E_p f)_\ast$. The result follows.
\end{proof}

In \cite{thomason1980cat}, Thomason proves that $\cat$ admits a closed model category structure which is Quillen equivalent to the usual model category structure of the category of simplicial sets. Recall that a functor $F\colon C\to D$ in $\cat$ is a weak equivalence if and only if the induced continuous map $BF\colon BC\to BD$ is a homotopy equivalence. Thomason also proves that the cofibrant objects of this model category structure are posets. 

Several authors study which posets are cofibrant objects of the Thomason model structure. Bruckner and Pegel prove in \cite{bruckner2016cofibrant} that various classes of posets are cofibrant objects and May, Stephan and Zakharevich give in \cite{may2017homotopy} a poset of six elements which is not cofibrant. Recall also that the double subdivision of a poset is a cofibrant object in $\cat$ (see \cite[Proposition 4.6]{thomason1980cat}).

\begin{definition}
Let $X$ be an Alexandroff T$_0$--space and let $A\subseteq X$ be such that for each connected component $C$ of $X$, $\#(A\cap C)=1$. Let $F$ be a T$_0$--space and let $\mathcal{U}_F\colon Q\aut(F)\to\aut(F)$ be the cofibrant replacement of $\aut(F)$ in $\cat$ (recall that $Q\aut(F)$ is a poset). Let $f,g\colon X \to Q\aut(F)$ be continuous maps.

We say that $f$ is \emph{equivalent} to $g$ if for each $x_0\in A$ there exists $\nu_{x_0}\in\Inn(\aut(F))$ such that 
\begin{displaymath}
(\mathcal{U}_F g)_\ast=\nu_{x_0} (\mathcal{U}_F f)_\ast\colon \pi_1(X,x_0) \to \pi_1(\aut(F),F)\cong\aut(F).
\end{displaymath}
\end{definition}

It is not difficult to verify that the definition of equivalent maps does not depend on the subset $A$.

\begin{theo}
Let $F$ be a T$_0$--space and let $\mathcal{U}_F\colon Q\aut(F)\to\aut(F)$ be the cofibrant replacement of $\aut(F)$ in $\cat$. 

Then, for every non-empty cofibrant object $X$ of $\cat$ there exists a canonical bijection between isomorphism classes of fiber bundles over $X$ with fiber $F$ and equivalence classes of continuous maps from $X$ to $Q\aut(F)$. 
\end{theo}

\begin{proof}
Let $X$ be a non-empty cofibrant object of $\cat$. Thus, $X$ is a poset. Let $A\subseteq X$ be such that for each connected component $C$ of $X$, $\#(A\cap C)=1$. Let $[\Fib_X(F)]$ denote the set of isomorphism classes of fiber bundles over $X$ with fiber $F$ and let $\mathcal{F}$ denote the set of equivalence classes of continuous maps from $X$ to $Q\aut(F)$. Let $\iota\colon \aut(F)\to \Top$ be the inclusion functor. 
Let $\xi\colon \mathcal{F} \to [\Fib_X(F)]$ be the function that sends the equivalence class of a continuous map $f\colon X \to Q\aut(F)$ to the pullback of the fiber bundle $\pi_{Q\aut(F)}^{\iota \mathcal{U}_F}\colon \int \iota\mathcal{U}_F \to Q\aut(F)$ along $f$.

We will prove first that $\xi$ is well-defined. Suppose that $f,g\colon X \to Q\aut(F)$ are equivalent maps. Then for each $x_0\in A$ there exists $\nu_{x_0}\in\Inn(\aut(F))$ such that $(\mathcal{U}_F g)_\ast=\nu_{x_0} (\mathcal{U}_F f)_\ast\colon \pi_1(X,x_0) \to \pi_1(\aut(F),F)\cong\aut(F)$. By \ref{prop_functor_to_Aut_F_natural_iso}, the canonical representation of the fiber bundle $\pi_{Q\aut(F)}^{\iota\mathcal{U}_F}$ is naturally isomorphic to the functor $\iota \mathcal{U}_F$. Thus, by \ref{prop_equivalences_isomorphic_pullback_bundles}, the pullback bundles $\xi(f)$ and $\xi(g)$ are isomorphic. Hence $\xi$ is well-defined.

Observe that injectivity of $\xi$ follows from \ref{prop_equivalences_isomorphic_pullback_bundles} and the argument of the previous paragraph.

Now we will prove that $\xi$ is surjective. Let $p\colon E\to X$ be a fiber bundle over $X$ with fiber $F$ and let $\D_p\colon X\to \Top$ be its canonical representation. Let $\E_p\colon X \to \aut(F)$ be a functor which satisfies that $\iota \E_p$ and $\D_p$ are naturally isomorphic. Since $X$ is cofibrant and $\mathcal{U}_F$ is a trivial fibration, there exists a functor $f\colon X \to Q\aut(F)$ such that $\mathcal{U}_F f= \E_p$. Note that $f$ can be regarded as a continuous map between Alexandroff T$_0$--spaces. Let $\pi=\pi_{Q\aut(F)}^{\iota \mathcal{U}_F}$ and let $p_f$ be the pullback of $\pi$ along $f$. By \ref{prop_coro_canonical_representation_of_pullback} the canonical representation $\D_{p_f}$ of $p_f$ is naturally isomorphic to $\D_{\pi} f$, which is naturally isomorphic to $\iota \mathcal{U}_F f = \iota \E_p$. And since $\iota \E_p$ and $\D_p$ are naturally isomorphic it follows that $p_f$ and $p$ are isomorphic fiber bundles by \ref{prop_canonical_representations_naturally_isomorphic}.
\end{proof}

\begin{rem}
The previous theorem and its proof show that if $F$ is a T$_0$--space and $\mathcal{U}_F\colon Q\aut(F)\to\aut(F)$ is the cofibrant replacement of $\aut(F)$ in $\cat$ then the fiber bundle $\pi_{Q\aut(F)}^{\mathcal{U}_F}\colon \int \mathcal{U}_F \to Q\aut(F)$ serves as a universal bundle with fiber $F$ for bundles over posets which are cofibrant objects in $\cat$.
\end{rem}

\section{Fiber bundles over Alexandroff spaces are Hurewicz fibrations}

In this section we will use theorem \ref{theo_fiber_bundles_give_functors} to prove that fiber bundles over Alexandroff spaces are Hurewicz fibrations. To this end we will prove the following lemma for which we need to recall that a covering space of an Alexandroff space is an Alexandroff space \cite[Proposition 2.1]{barmak2016note} and that if $p\colon E\to B$ is a covering map between Alexandroff spaces then for every $e\in E$ the restriction $p|_{U_e}\colon U_e\to U_{p(e)}$ is a homeomorphism (see proof of \cite[Proposition 2.9]{barmak2016note}).

If $X$ and $Y$ are topological spaces, $K$ is a compact subspace of $X$ and $U$ is an open subset of $Y$ we denote $W(K,U)=\{f\colon X\to Y \st f \textnormal{ is a continuous map and }f(K)\subseteq U\}$.

\begin{lemma} \label{lemma_path_space_locally_connected}
Let $B$ be an Alexandroff space and let $\gamma\colon I \to B$ be a continuous map. Then there exist an open subset $V\subseteq B^I$ such that $\gamma\in V\subseteq W(\{0\},U_{\gamma(0)})\cap W(\{1\},U_{\gamma(1)})$ satisfying that for all $\alpha\in V$ and for all paths $\eta_0,\eta_1 \in B^I$ such that $\eta_0(0)=\gamma(0)$, $\eta_0(1)=\alpha(0)$, $\eta_0(I)\subseteq U_{\gamma(0)}$, $\eta_1(0)=\alpha(1)$, $\eta_1(1)=\gamma(1)$ and $\eta_1(I)\subseteq U_{\gamma(1)}$ it holds that $\gamma\simp \eta_0\ast \alpha \ast \eta_1$.
\end{lemma}

\begin{proof}
Without loss of generality we may assume that $B$ is connected, since we may replace $B$ with the connected component of $B$ that contains $\gamma(0)$ (note that, if $C$ is a connected component of $B$, then $C$ is an open subset of $B$ and thus $C^I=W(I,C)$ is an open subset of $B^I$).

Let $\widetilde B$ be the universal cover of $B$ and let $p\colon \widetilde B \to B$ be the associated covering map. Let $\widetilde b\in p^{-1}(\gamma(0))$ and let $\widetilde \gamma \in \widetilde B^I$ be the unique lift of $\gamma$ such that $\widetilde \gamma(0)=\widetilde b$. Let $p_0\colon U_{\widetilde \gamma(0)} \to U_{\gamma(0)}$ and $p_1\colon U_{\widetilde \gamma(1)} \to U_{\gamma(1)}$ be restrictions of $p$. Note that $p_0$ and $p_1$ are homeomorphisms.

Let $B^I \times_p \widetilde B$ be the pullback of the diagram $B^I \overset{\ev_0}{\longrightarrow} B \overset{p}{\longleftarrow} \widetilde B$. As usual, we regard $B^I \times_p \widetilde B$ as a subspace of the product space $B^I \times \widetilde B$. Let $\lambda\colon B^I \times_p \widetilde B \to \widetilde B^I$ be the unique path lifting function for $p$. Since $\lambda^{-1}(\ev_1^{-1}(U_{\widetilde \gamma(1)}))$ is an open neighbourhood of $(\gamma,\widetilde b)$ in $B^I \times_p \widetilde B$ there exists an open subset $V'\subseteq B^I$ such that $(\gamma,\widetilde b) \in V'\times U_{\widetilde b}$ and $(V'\times U_{\widetilde b})\cap (B^I \times_p \widetilde B)\subseteq \lambda^{-1}(\ev_1^{-1}(U_{\widetilde \gamma(1)}))$. Let $V= V'\cap W(\{0\},U_{\gamma(0)})\cap W(\{1\},U_{\gamma(1)})$. Note that $V$ is an open neighbourhood of $\gamma$.

We will prove now that the open subset $V$ satisfies the required condition. Let $\alpha\in V$ and let $\eta_0,\eta_1 \in B^I$ such that $\eta_0(0)=\gamma(0)$, $\eta_0(1)=\alpha(0)$, $\eta_0(I)\subseteq U_{\gamma(0)}$, $\eta_1(0)=\alpha(1)$, $\eta_1(1)=\gamma(1)$ and $\eta_1(I)\subseteq U_{\gamma(1)}$. Let $\widetilde \eta_0 = p_0^{-1}\eta_0$, let $\widetilde \eta_1 = p_1^{-1}\eta_1$ and let $\widetilde \alpha$ be the unique lift of $\alpha$ such that $\widetilde \alpha(0)= p_0^{-1}(\alpha(0))$. Note that $\widetilde \eta_0 (1) = p_0^{-1}\eta_0(1) = p_0^{-1}\alpha(0) = \widetilde \alpha(0)$. On the other hand, observe that $(\alpha,p_0^{-1}(\alpha(0)))\in (V\times U_{\widetilde b}) \cap (B^I \times_p \widetilde B) \subseteq \lambda^{-1}(\ev_1^{-1}(U_{\widetilde \gamma(1)}))$. Hence $\widetilde \alpha(1)\in U_{\widetilde \gamma(1)}$ and since $p\widetilde \alpha(1)=\alpha(1)=\eta_1(0)$ it follows that $\widetilde \alpha(1)= \widetilde \eta_1(0)$. Therefore, $(\widetilde \eta_0 \ast \widetilde \alpha ) \ast \widetilde \eta_1$ is a well-defined path in $\widetilde B$ from $\widetilde \eta_0(0)$ to $\widetilde \eta_1(1)$. 
Note that $\widetilde \eta_0(0)=p_0^{-1}\eta_0(0)=p_0^{-1}\gamma(0)=\widetilde b$ and similarly $\widetilde \eta_1(1)=\widetilde \gamma(1)$. Thus, $(\widetilde \eta_0 \ast \widetilde \alpha ) \ast \widetilde \eta_1$ and $\widetilde \gamma$ are paths in $\widetilde B$ from $\widetilde b$ to $\widetilde \gamma(1)$. Since $\widetilde B$ is simply connected we obtain that $\widetilde \gamma \simp \widetilde \eta_0 \ast \widetilde \alpha \ast \widetilde \eta_1$. Therefore, $\gamma\simp \eta_0\ast \alpha \ast \eta_1$.
\end{proof}

\begin{theo} 
Let $B$ be an Alexandroff space and let $p\colon E\to B$ be a fiber bundle over $B$. Then $p$ is a Hurewicz fibration.
\end{theo}

\begin{proof}
By \ref{theo_fiber_bundles_give_functors}, there exists a morphism-inverting functor $\D\colon B\to\Top$ such that the projection $\pi^{\D}_B\colon \int \D \to B$ is a fiber bundle isomorphic to $p$. Clearly, it suffices to prove that $\pi^{\D}_B$ is a Hurewicz fibration.

Let $\loc\D\colon \loc B\to\Top$ be the functor induced by $\D$ and let $\overline{\D}$ be the composition $\Pi_1(B)\cong \loc B \overset{\loc\D}{\longrightarrow} \Top$. Let $B^I \times_{\pi^{\D}_B} \int \D$ be the pullback of the diagram $B^I \overset{\ev_0}{\longrightarrow} B \overset{\pi^{\D}_B}{\longleftarrow} \int \D$. We regard $B^I \times_{\pi^{\D}_B} \int\D$ as a subspace of the product space $B^I \times \int\D$.

Let $\Lambda\colon (B^I \times_{\pi^{\D}_B} \int \D)\times I \to \int \D$ be defined by $\Lambda(\gamma,b,x,t)=(\gamma(t),\overline{\D}([\gamma_{[0,t]}])(x))$ and let $\Lambda^\sharp\colon B^I \times_{\pi^{\D}_B} \int \D \to (\int \D)^I$ be the map induced by $\Lambda$ by the exponential law. We will prove that $\Lambda^\sharp$ is a path-lifting function for $\pi^{\D}_B$. Note that $\Lambda(\gamma,b,x,0)=(\gamma(0),\overline{\D}([\gamma_{[0,0]}])(x))=(b,x)$ and that $\pi^{\D}_B\Lambda(\gamma,b,x,t)=\gamma(t)$ for all $(\gamma,b,x)\in B^I \times_{\pi^{\D}_B} \int \D$ and $t\in I$. Thus, it remains to prove that $\Lambda$ is a continuous map.

Let $\beta\in B$ and let $W$ be an open subset of $\D(\beta)$. Let $(\gamma,b,x,t)\in\Lambda^{-1}(J_\D(\beta,W))$. Hence $\gamma(t)\leq \beta$ and $\D(\gamma(t)\leq\beta)(\overline{\D}([\gamma_{[0,t]}])(x))\in W$. Let $W_0=\overline{\D}([\gamma_{[0,t]}])^{-1}\D(\gamma(t)\leq\beta)^{-1}(W)$. Note that $W_0$ is an open neighbourhood of $x$ in $\D(\gamma(0))$.

By the previous lemma there exists an open subset $V\subseteq B^I$ such that $\gamma_{[0,t]}\in V\subseteq W(\{0\},U_{\gamma(0)})\cap W(\{1\},U_{\gamma(t)})$ satisfying that for all $\alpha\in V$ and for all paths $\eta_0,\eta_1 \in B^I$ such that $\eta_0(0)=\gamma(0)$, $\eta_0(1)=\alpha(0)$, $\eta_0(I)\subseteq U_{\gamma(0)}$, $\eta_1(0)=\alpha(1)$, $\eta_1(1)=\gamma(t)$ and $\eta_1(I)\subseteq U_{\gamma(t)}$ it holds that $\gamma_{[0,t]}\simp \eta_0\ast \alpha \ast \eta_1$.

Let $\nu\colon B^I \to B^I$ be defined by $\nu(\sigma)=\sigma_{[0,t]}$. It is not difficult to prove that $\nu$ is a continuous map. Let $L\subseteq I$ be an open interval such that $t\in L \subseteq \overline{L} \subseteq \gamma^{-1}(U_{\gamma(t)})$. Let 
$A= \nu^{-1}(V)\cap W(\overline{L},U_{\gamma(t)})$, let $N_0=A\times J_\D(\gamma(0),W_0)$ and let $N=N_0 \cap (B^I \times_{\pi^{\D}_B} \int \D)$. Note that $N$ is an open subset of $B^I \times_{\pi^{\D}_B} \int \D$.

We will prove that $(\gamma,b,x,t)\in N\times L \subseteq \Lambda^{-1}(J_\D(\beta,W))$. Clearly $\gamma\in A$ and $(b,x)\in J_\D(\gamma(0),W_0)$ since $\gamma(0)=b$. Hence $(\gamma,b,x,t)\in N\times L$. Now let $(\gamma',b',x',t')\in N\times L$. Note that $\gamma'(t')\leq \gamma(t)\leq\beta$ since $t'\in L$. If $t'\geq t$ let $\xi=\gamma'_{[t,t']}$, otherwise let $\xi=\overline{\gamma'_{[t',t]}}$. In any case, $\xi$ is a path in $B$ from $\gamma'(t)$ to $\gamma'(t')$ whose image is contained in $U_{\gamma(t)}$ since $\gamma'\in W(\overline{L},U_{\gamma(t)})$. Note that $\gamma'(0)\leq\gamma(0)$ since $\gamma'_{[0,t]}\in V$. Let $\eta_0=\eta(\gamma'(0)\leq\gamma(0))$.

We have that
\begin{align*}
\gamma'_{[0,t']}\ast \eta(\gamma'(t')\leq\beta) & \simp \eta_0 \ast \overline{\eta_0} \ast \gamma'_{[0,t]}\ast \xi \ast \eta(\gamma'(t')\leq\gamma(t)) \ast \eta(\gamma(t)\leq\beta) \simp \\ 
& \simp \eta_0 \ast \gamma_{[0,t]} \ast \eta(\gamma(t)\leq\beta) 
\end{align*}
since $\eta_0(I)\subseteq U_{\gamma(0)}$ and $(\xi \ast \eta(\gamma'(t')\leq\gamma(t)))(I)\subseteq U_{\gamma(t)}$.
Thus
\begin{align*}
\D(\gamma'(t')\leq \beta)(\overline{\D}([\gamma'_{[0,t']}])(x')) &= \overline{\D}([\gamma'_{[0,t']}\ast \eta(\gamma'(t')\leq\beta)  ])(x') = \\
&= \overline{\D}([ \eta_0 \ast \gamma_{[0,t]} \ast \eta(\gamma(t)\leq\beta)  ])(x') = \\ 
&= \D(\gamma(t)\leq\beta) \overline{\D}([\gamma_{[0,t]}]) \D(\gamma'(0)\leq\gamma(0))(x') 
\end{align*}
and since $\D(\gamma'(0)\leq\gamma(0))(x')\in W_0$ it follows that $\D(\gamma'(t')\leq \beta)(\overline{\D}([\gamma'_{[0,t']}])(x'))\in W$. Thus, $\Lambda(\gamma',b',x',t')\in J_\D(\beta,W)$. It follows that $\Lambda^{-1}(J_\D(\beta,W))$ is an open subset of $(B^I \times_{\pi^{\D}_B} \int \D)\times I $. Therefore, $\Lambda$ is a continuous map.
\end{proof}

\section{The topological Grothendieck construction as an equivalence of categories}

In this section we will define a suitable category of functors and prove that is equivalent to the category of fiber bundles over a fixed Alexandroff space by means of the topological Grothendieck construction. To this end, we need to define a convenient notion of arrows between functors which relies on the following definition.

\begin{definition}
Let $X$ and $Y$ be topological spaces, let $\T_Y$ be the topology of $Y$ and let $C(X,Y)$ be the set of continuous maps from $X$ to $Y$. We define a preorder $\preceq$ in $C(X,Y)$ by 
\begin{displaymath}
f\preceq g \Leftrightarrow \forall\, V \in \T_Y, g^{-1}(V) \subseteq f^{-1}(V) .
\end{displaymath}
\end{definition}

Observe that, with the notations of the previous definition, if $f\preceq g$ and $g\preceq f$ then $\Ko(f)=\Ko(g)$ by \ref{lemma_Kolmogorov_quotient}.

The following proposition states that this preorder coincides with the pointwise preorder in the case that the codomain is an Alexandroff space.

\begin{prop}
Let $X$ be a topological space and let $Y$ be an Alexandroff space. Let $f,g\colon X\to Y$ be continuous maps. Then $f\preceq g$ if and only if $f(x)\leq g(x)$ for all $x\in X$.
\end{prop}

\begin{proof}
Suppose that $f\preceq g$ and let $x\in X$. Then $x\in g^{-1}(U_{g(x)})\subseteq f^{-1}(U_{g(x)})$ and thus $f(x)\leq g(x)$. Conversely, suppose that $f\leq g$ and let $V$ be an open subset of $Y$. Let $x\in g^{-1}(V)$. Then $g(x)\in V$ and since $f(x)\leq g(x)$ we obtain that $f(x)\in V$. Hence, $g^{-1}(V) \subseteq f^{-1}(V)$.
\end{proof}

\begin{rem}
Let $X$ and $Y$ be topological spaces, let $\T_X$ be the topology of $X$ and let $\T_Y$ be the topology of $Y$. If $f\colon X\to Y$ is a continuous map then $f$ induces an order-preserving map $\widehat{f}\colon (\T_Y,\subseteq) \to (\T_X,\subseteq)$ defined by $\widehat{f}(V)=f^{-1}(V)$.

Hence, if $f,g\colon X\to Y$ are continuous maps then $f\preceq g$ if and only if $\widehat{g}\leq\widehat{f}$.
\end{rem}

\begin{definition} \label{def_weak_natural_transformation}
Let $B$ be an Alexandroff space and let $C,D\colon B\to \Top$ be functors. Let $\{\theta_b\colon C(b)\to D(b)\}_{b\in B}$ be a collection of continuous maps. We say that $\{\theta_b\}_{b\in B}$ is a \emph{weak natural transformation from $C$ to $D$} if for all $b_1,b_2\in B$ such that $b_1\leq b_2$ we have that $D(b_1\leq b_2)\theta_{b_1}\preceq \theta_{b_2}C(b_1\leq b_2)$.
\end{definition}

It is not difficult to verify that for any Alexandroff space $B$ there exists a category (that will be denoted by $\Top^B_w$) whose objects are the functors from $B$ to $\Top$ and whose morphisms are the weak natural transformations, with composition defined in a similar way as vertical composition of natural transformations.

The following proposition extends \ref{rem_int_is_functor}.

\begin{prop}
Let $B$ be an Alexandroff space. The topological Grothendieck construction induces a functor $\int\colon\Top^B_w\to\Top/B$ which sends each object $C$ of $\Top^B_w$ to the map $\pi_B^C\colon \int C\to B$ and each weak natural transformation $\tau=\{\alpha_b\}_{b\in B}\in \Hom_{\Top^B_w}(C,D)$ to the map $\int\tau\colon \int C \to \int D$ defined by $\int\tau(b,x)=(b,\alpha_b(x))$.
\end{prop}

\begin{proof}
Let $\tau=\{\alpha_b\}_{b\in B}\in \Hom_{\Top^B_w}(C,D)$. We will prove that $\int\tau$ is a continuous map. For simplicity, let $\alpha=\int\tau$. Let $b_1\in B$ and let $V$ be an open subset of $D(b_1)$. We will prove that $\alpha^{-1}(J_D(b_1,V))$ is an open subset of $\int C$. Let $(b_0,x)\in \alpha^{-1}(J_D(b_1,V))$. Then $(b_0,\alpha_{b_0}(x))\in J_D(b_1,V)$ and thus $b_0\leq b_1$ and $D(b_0\leq b_1)(\alpha_{b_0}(x))\in V$. Hence $x\in (D(b_0\leq b_1)\alpha_{b_0})^{-1}(V)$. Let $W=D(b_0\leq b_1)^{-1}(V)$ and let $U=\alpha_{b_0}^{-1}(W)$. Note that $(b_0,x)\in J_C(b_0,U)$. We will prove that $J_C(b_0,U)\subseteq \alpha^{-1}(J_D(b_1,V))$. Let $(\beta,z)\in J_C(b_0,U)$. Then $\beta\leq b_0$ and $C(\beta\leq b_0)(z)\in U$. Thus $z\in C(\beta\leq b_0)^{-1} \alpha_{b_0}^{-1} (W)$ and since $\{\alpha_b\}_{b\in B}$ is a weak natural transformation we obtain that $z\in (D(\beta\leq b_0)\alpha_{\beta})^{-1}(W)$. Hence $D(\beta\leq b_0)\alpha_{\beta}(z)\in W$ and thus $D(b_0\leq b_1)D(\beta\leq b_0)\alpha_{\beta}(z)\in V$. Then $\alpha(\beta,z)=(\beta,\alpha_\beta(z))\in J_D(b_1,V)$. Therefore, $\alpha$ is a continuous map. The result follows.
\end{proof}

\begin{prop} \label{prop_hom_correspondence}
Let $B$ be an Alexandroff space and let $C,D\colon B\to \Top$ be functors. Then the functor $\int\colon\Top^B_w\to\Top/B$ induces a one-to-one correspondence between $\Hom_{\Top^B_w}(C,D)$ and $\Hom_{\Top/B}(\pi_B^C,\pi_B^D)$.
\end{prop}

\begin{proof}
Let $\omega\colon \Hom_{\Top^B_w}(C,D) \to \Hom_{\Top/B}(\pi_B^C,\pi_D^C)$ be defined by $\omega(\tau)=\int\tau$. We will prove that $\omega$ admits an inverse.

For each $\alpha\in \Hom_{\Top/B}(\pi_B^C,\pi_B^D)$ and for each $b\in B$ we define $\alpha_b\colon C(b) \to D(b)$ as the only map such that $\alpha(b,x)=(b,\alpha_b(x))$ for all $x\in C(b)$. Note that $\alpha_b$ is a continuous map by \ref{rem_embedding}.

We will prove that, for all $\alpha\in \Hom_{\Top/B}(\pi_B^C,\pi_B^D)$, the collection $\{\alpha_b\}_{b\in B}$ is a weak natural transformation from $C$ to $D$. Let $\alpha\in \Hom_{\Top/B}(\pi_B^C,\pi_B^D)$ and let $b_1,b_2\in B$ such that $b_1\leq b_2$. Let $V$ be an open subset of $D(b_2)$. We have to prove that $(\alpha_{b_2}C(b_1\leq b_2))^{-1}(V) \subseteq  (D(b_1\leq b_2)\alpha_{b_1})^{-1}(V)$. Let $x\in (\alpha_{b_2}C(b_1\leq b_2))^{-1}(V)$ and let $y=C(b_1\leq b_2)(x)$. Then $\alpha_{b_2}(y)\in V$ and thus $(b_2,y)\in\alpha^{-1}(J_D(b_2,V))$. Hence, there exist $b_3\in B$ and an open subset $U\subseteq C(b_3)$ such that $(b_2,y)\in J_C(b_3,U)\subseteq \alpha^{-1}(J_D(b_2,V))$. Note that $(b_1,x)\in J_C(b_3,U)$ since $b_1\leq b_2\leq b_3$ and $C(b_1\leq b_3)(x)=C(b_2\leq b_3)(y)\in U$. Thus, $(b_1,\alpha_{b_1}(x))=\alpha(b_1,x)\in J_D(b_2,V)$. Hence, $D(b_1\leq b_2)\alpha_{b_1}(x)\in V$.

Let $\rho\colon \Hom_{\Top/B}(\pi_B^C,\pi_D^C) \to \Hom_{\Top^B_w}(C,D)$ be defined by $\rho(\alpha)=\{\alpha_b\}_{b\in B}$. Clearly, $\rho$ and $\omega$ are mutually inverse functions.
\end{proof}

Observe that the functor $\int\colon\Top^B_w\to\Top/B$ does not yield an equivalence of categories since it is not essentially surjective, as it was shown in \ref{ex_topological_Grothendieck_construction_is_not_essentially_surjective}. However, it induces equivalences of categories between certain full subcategories of $\Top^B_w$ and categories of fiber bundles over $B$, which are stated in the following theorem.

\begin{theo} \label{theo_equivalence_categories}\ 
\begin{enumerate}
\item Let $B$ be an Alexandroff space and let $F$ be a topological space. Let $\aut(F)^B$ denote the full subcategory of $\Top^B_w$ whose objects are the functors from $B$ to $\aut(F)$ and let $\Fib_B(F)$ denote the category of fiber bundles over $B$ with fiber $F$. Then, the functor $\int\colon\Top^B_w\to\Top/B$ induces an equivalence of categories between $\aut(F)^B$ and $\Fib_B(F)$.

\item Let $B$ be a connected non-empty Alexandroff space. Let $(\Top^B_w)'$ denote the full subcategory of $\Top^B_w$ whose objects are the morphism-inverting functors and let $\Fib_B$ denote the category of fiber bundles over $B$. Then, the functor $\int\colon\Top^B_w\to\Top/B$ induces an equivalence of categories between $(\Top^B_w)'$ and $\Fib_B$.
\end{enumerate}
\end{theo}

\begin{proof} \ 
By \ref{coro_functor_to_Aut_F_fiber_bundle} the functor $\int\colon\Top^B_w\to\Top/B$ can be restricted to a functor $\aut(F)^B \to \Fib_B(F)$, which is fully faithful by \ref{prop_hom_correspondence} and essentially surjective by \ref{theo_fiber_bundles_give_functors}, \ref{lemma_top0_vs_autF} and \ref{rem_int_is_functor}. This proves (1).

Similarly, by \ref{coro_int_D_is_fiber_bundle} the functor $\int\colon\Top^B_w\to\Top/B$ can be restricted to a functor $(\Top^B_w)'\to\Fib_B$ which is fully faithful by \ref{prop_hom_correspondence} and essentially surjective by \ref{theo_fiber_bundles_give_functors}. This proves (2).
\end{proof}

\bibliographystyle{acm}
\bibliography{ref_fiber_bundles}

\end{document}